\documentclass[12pt]{article}
\usepackage[a4paper, total={6in, 8in}]{geometry}
\usepackage{amsmath,amssymb,amsthm,bbm,float,graphicx,geometry,lmodern,mathtools,parskip,setspace,subcaption,subcaption}
\usepackage[colorlinks]{hyperref}
\usepackage{mathrsfs}
\usepackage{enumerate}
\usepackage{tikz}
\usetikzlibrary{backgrounds}
\usepackage{caption}
\usepackage{graphicx}
\usepackage{cleveref}
\graphicspath{{Images/}}
\allowdisplaybreaks
\DeclarePairedDelimiter{\set}{\lbrace}{\rbrace}
\newcommand{\N}{\mathbb{N}}

\theoremstyle{definition}
\newtheorem{theorem}{Theorem}[section]
\newtheorem{lemma}[theorem]{Lemma}

\renewcommand{\P}{\mathbb{P}}
\newcommand{\x}{\mathbf{x}}
\newcommand{\y}{\mathbf{y}}

\newcommand{\G}{\mathscr{G}}

\newcommand{\C}{\mathscr{C}}
\newcommand{\X}{\mathcal{X}}
\newcommand{\E}{\mathbb{E}}
\newcommand{\0}{\mathbf{0}}
\newcommand{\R}{\mathbb{R}}
\renewcommand{\d}{\mathrm{d}}
\newcommand{\z}{\mathbf{z}}
\newcommand{\w}{\mathbf{w}}
\renewcommand{\v}{\mathbf{v}}
\renewcommand{\u}{\mathbf{u}}
\newcommand{\gp}{g^{\text{pa}}}
\newcommand{\heap}[2]  {\genfrac{}{}{0pt}{}{#1}{#2}}
\usepackage[style = numeric, sorting=nyt, url = false, abbreviate=false, maxbibnames=9, sortcites=true, backend = biber]{biblatex}
\renewbibmacro{in:}{}
\bibliography{agedependent.bib}

\author{
Peter Gracar\thanks{Department of Mathematics, University of Cologne, Weyertal 86-90, 50931 K\"oln, Germany.}\\pgracar@math.uni-koeln.de\\
\and
Lukas L\"{u}chtrath\footnotemark[1]\\l.luechtrath@uni-koeln.de\\
\and
Peter M\"orters\footnotemark[1]\\moerters@math.uni-koeln.de\\
}
\title{Percolation phase transition in weight-dependent random connection models}
\usepackage{etoolbox}
\ifundef{\abstract}{}{\patchcmd{\abstract}%
{\quotation}{\quotation\noindent\ignorespaces}{}{}}
\begin{document}
\maketitle 
\begin{spacing}{0.9}
\begin{abstract} We investigate spatial random graphs defined on the points of a Poisson process in $d$-dimensional space, which combine scale-free degree distributions and long-range effects. Every Poisson point is assigned an independent weight. Given the weight and position of the points, we form an edge between any pair of points independently with a probability depending on the two weights of the points and their distance. Preference is given to short edges and connections to vertices with large weights. 
We characterize the parameter regime where there is a non-trivial percolation phase transition and show that it depends not only on the power-law exponent of the degree distribution but also on a geometric model parameter. We apply this result to characterize robustness of  age-based spatial preferential attachment networks. \\
\\
\footnotesize{{\textbf{AMS-MSC 2010}: 60K35, 05C80}
}
\end{abstract}
\end{spacing}
\begin{spacing}{1}
\newpage

\section{Introduction and statement of results}\label{SecIntro}

\subsection*{Motivation}
In classical continuum percolation theory a graph is built with a Poisson point process in $\R^d$ as the vertex set. Two points are connected by an edge if their Euclidean distance is below a fixed or variable threshold.  Assuming the resulting graph has an infinite component, one asks whether there exists an infinite component in the percolated graph where every edge is independently removed with probability $1-p$, respectively retained with probability $p$. We say that the graph has a \emph{percolation phase transition} if there is a critical probability $p_c>0$ such that, almost surely,  if $p<p_c$  there is no infinite component, and if $p>p_c$ there exists an infinite component in the percolated graph. It is known that there exists a percolation phase transition for the fixed threshold model in $\R^d$, often called the Boolean model, and for variable threshold models where the threshold is the sum of  independent radii with finite $d$th moment associated with the points~\cite{gouere2008, gouere2009}. The result also extends to long-range percolation models, where the probability that two points are connected is a decreasing function of their distance, see \cite{penrose91, Meester96}. 
\medskip
  
By contrast, the continuous version of the scale-free percolation model of van der Hofstad, Hooghiemstra and Deijfen
\cite{DeijfenHofstadHooghiemstra2013} does not have a percolation phase transition if the power-law exponent  satisfies $\tau<3$, see for example \cite{heydenreich2017, Deprez2018}. In fact, for many graphs combining scale-free degree distributions and long-range effects the problem of 
existence of a percolation phase transition
is open. This includes, for example, models where the connection probability of two points is a decreasing function of the ratio of their distance and the sum or maximum of their radii. In this paper we look at a broad class of such graphs, the \emph{weight-dependent random connection models}, and characterize  the parameter regimes where there is a percolation phase transition. Other than in the scale-free percolation model, in this class a subcritical phase can only fail to exist if there is sufficiently small power-law exponent combined with a strong long-range effect. The weight-dependent random connection models include the weak local limits of the age-based preferential attachment model introduced 
in~\cite{GracarEtAl2019}. We use this result to characterize the regimes when these network models 
are robust under random removal of edges
offering new insight into the notoriously difficult topic of spatial preferential attachment networks, see~\cite{JacobMoerters2017}.
\medskip

\pagebreak[3]

\subsection*{Framework}

We introduce the weight-dependent random connection model as in~\cite{gracar2019recurrence}. The vertex set of the graph $\G$ is a Poisson point process of unit intensity on $\mathbb{R}^d\times (0,1]$. We think of a Poisson point $\mathbf{x}=(x, t)$ as a \emph{vertex} at \emph{position} $x$ with \emph{weight} $t^{-1}$. Two vertices $\x$ and $\y$ are connected by an edge in $\G$ independently of any other (possible) edge with probability $\varphi(\x,\y)$. Here, $\varphi$ is a connectivity function
	\[\varphi: (\mathbb{R}^d\times (0,1])\times(\mathbb{R}^d\times(0,1])\to [0,1],\]
of the form 
	\[\varphi(\x,\y)=\varphi((x,t),(y,s))=\rho(g(t,s)|x-y|^d)\]
for a non-increasing, integrable \emph{profile function} $\rho:\mathbb{R}_+\to[0,1]$ and a function $g\colon(0,1)\times(0,1)\to\mathbb{R}_+$, which is symmetric and non-decreasing in both arguments. Hence, we give preference to short edges or edges that are connected to vertices with large weights. We also assume (without loss of generality) that 
	\begin{equation}
		\int_{\mathbb{R}^d}\ \rho(|x|^d) \,  \d x =1. \label{IntegrabilityCond}
	\end{equation}
Then, the degree distribution of a vertex only depends on the function $g$. However, the profile function controls the intensity of long edges in the graph. 
\smallskip

We next give explicit examples for the  function $g$ we will focus on throughout the paper. We define the functions in terms of two parameters $\gamma\in(0,1)$ and $\beta\in(0,\infty)$. The parameter $\gamma$ describes the strength of the influence of the vertices' weights on the connection probability; the larger $\gamma$, the stronger the preference of connecting to vertices with large weight.  All kernel functions we consider lead to models that are \emph{scale-free} with power law exponent 
	\[\tau = 1+\frac{1}{\gamma},\]
see \cite{gracar2019recurrence, GracarEtAl2019}. {In particular,} all graphs are locally finite, i.e.\ every vertex has finite degree. The parameter $\beta$ is used to control the edge density, i.e.\ increasing $\beta$ increases the expected number of edges connected to a typical vertex \cite{GracarEtAl2019}. Our focus is on the following three functions, for further examples, see \cite{gracar2019recurrence}. 

\begin{itemize}
	\item The \emph{sum kernel}, defined as
		\[g^\text{sum}(s,t)=\beta^{-1} (s^{-\gamma/d}+t^{-\gamma/d})^{-d}.\]
	The interpretation of $(\beta a s^{-\gamma})^{1/d}, (\beta a t^{-\gamma})^{1/d}$ as random radii together with $\rho(r)=\mathbbm{1}_{[0,a]}(r)$ leads to the Boolean model in which two vertices are connected by an edge when their associated balls intersect.\pagebreak[3]
	\item The \emph{min kernel}, defined as
		\[g^\text{min}(s,t)=\beta^{-1}(s\wedge t)^\gamma.\]
		Here, in the case of an indicator profile function as above, two vertices are connected by an edge when one of them lies inside the ball associated with the other one.
	As $2^{-d} g^{\text{min}}\leq g^\text{sum}\leq g^\text{min}$ 
	the min kernel and the sum kernel show qualitatively similar behaviour.
	\item The \emph{preferential attachment kernel}, 
	defined as 
		\begin{equation}
			g^\text{pa}(s,t) = \beta^{-1}(s\vee t)^{1-\gamma}(s\wedge t)^\gamma. \label{PAKernel}
		\end{equation}
	It gives rise to the \emph{age-dependent random connection model} introduced by Gracar et al. \cite{GracarEtAl2019}. This model is the weak local limit of the age-based spatial preferential attachment model which is an approximation of the spatial preferential attachment model introduced by Jacob and M\"orters \cite{JacobMoerters2015}. 
\end{itemize}

As we want to study the influence of long-range effects on the percolation problem, we focus primarily on profile functions that are 
\emph{regularly varying} with index $-\delta$ for some $\delta>1$, that is
	\begin{equation}
	\lim_{r\uparrow\infty} \frac{\rho(cr)}{\rho(r)} = c^{-\delta} \quad\mbox{ for all } c\geq 1. \label{RegularVarying}
	\end{equation}	
A comparison argument can be  used to derive the behaviour of profile functions with lighter tails (including those
with bounded support) from a limit $\delta\uparrow \infty$. 
\smallskip

We fix one of the kernels above, as well as $\gamma$, $\beta$ and $\delta$.  Let $p\in[0,1]$ and perform Bernoulli bond percolation
with retention parameter $p$ on the graph $\G$, i.e., every edge of~$\G$ remains intact independently with probability $p$, or is removed with probability~$1-p$. We denote the graph we obtain by $\G^p$ 
and ask whether there exists an infinite cluster, or equivalently an infinite self-avoiding path, in $\G^p$. If so, 
we say that the graph \emph{percolates}. 
We define the \emph{critical percolation parameter} $p_c$ as the infimum of all parameters $p\in[0,1]$ such that the percolation probability is positive. By the Kolmogorov 0-1--law, 
for all $1\geq p>p_c$ the graph percolates and for all $0\leq p<p_c$ the graph does not percolate, almost surely.
We call the parameter range $(p_c,1]$ the 
\emph{supercritical phase} and $[0,p_c)$ the
\emph{subcritical phase}. 
\pagebreak[3]

\subsection*{Main result: Percolation phase transition}

Our main result characterizes the parameter regime where there is a percolation phase transition
in the weight-dependent random connection model.
\bigskip

 \begin{theorem}[Percolation phase transition] \label{ThmPercolationPhase} 
Suppose $\rho$ satisfies \eqref{RegularVarying} for some $\delta>1$. Then, for the weight-dependent random connection model with preferential attachment kernel, sum kernel or min kernel and parameters $\beta>0$, 
$0<\gamma<1$, we have that
 	\begin{enumerate}[(a)]
 		\item if $\gamma<\frac{\delta}{\delta+1}$, then $p_c>0$. \smallskip
		
 		\item If $\gamma >\frac{\delta}{\delta+1}$, then $p_c=0$.
 	\end{enumerate}
 \end{theorem}
 \bigskip
 
\textbf{Remarks:}
 \begin{enumerate}[(i)] 
 \item We obtain the following estimates for $p_c$ from our proof.
 		\begin{itemize}
 			\item if $\gamma<\frac{1}{2}$, then $p_c\geq \frac{1-2\gamma}{4\beta}$. 
 			\item if $\rho(x)\leq Ax^{-{\delta}}$ for $A>1$, and $\frac{1}{2}\leq \gamma <\frac{\delta}{\delta+1}$, then 
 			\[p_c>\tfrac{1}{A}\big(\tfrac{d(\delta(1-\gamma)-\gamma)(\delta-1)}{2^{d\delta+4}J(d)\beta\delta }\big)^{\delta},\] 
 	where $J(d)=\prod_{j=0}^{d-2}\int_0^\pi \sin^j(\alpha_j)\d \alpha_j$ is the Jacobian of the $d$-dimensional sphere coordinates.
 	 \end{itemize}
	
	\item 	If $\gamma<\frac{\delta}{\delta+1}$ one can follow the argument for long-range percolation, see \cite{newman1986}, and check that if $d\geq 2$ or if $d=1$ and $\delta<2$ there 
	exists $\beta_c<\infty$ such that the graph percolates for all $\beta>\beta_c$ and fixing such a $\beta$ we then get $p_c<1$.
	
	\item If $\gamma =\frac{\delta}{\delta+1}$,  {we do not expect a} universal result, i.e.\ it depends on the exact form of the
	kernel $g$ and the profile $\rho$ whether $p_c=0$ or not.
	
	\item {A variant} of our arguments show that  if $\gamma<\frac{\delta}{\delta+1}$ and either $d\geq 2$ or $d=1$ and $\delta<2$, 
	there exists $0<\beta_c<\infty$ such that
	there does not exist an infinite component in $\G$ if $\beta<\beta_c$ but it does exist if $\beta>\beta_c$. 
	{By scaling the Poisson process} we see that, if $\gamma$ and $\delta$ are as above, $\beta>0$ is 
	fixed and the intensity of the Poisson process is variable, say $\lambda>0$, there exists  $0<\lambda_c<\infty$ such that
	there does not exist an infinite component in $\G$ if $\lambda<\lambda_c$ but it does exist if $\lambda>\lambda_c$. If however
	$\gamma>\frac{\delta}{\delta+1}$ there exists an infinite component in $\G$ regardless of the values of $\lambda, \beta>0$. 
	
	
	\item { To understand the occurrence of the critical value \smash{$\gamma=\frac{\delta}{\delta+1}$} the calculation in Lemma~\ref{LemTwoConnection} is key. There it is shown   that for $\gamma<\frac{\delta}{\delta+1}$ and small $p$ the probability that two sufficiently distant vertices are connected using an intermediate vertex of smaller weight is smaller than the probability of existence of a direct edge. If  $\gamma>\frac{\delta}{\delta+1}$  a converse statement holds, and it is more likely that two vertices of large weight are connected by an intermediate vertex of small weight. The corresponding strategy enters into the construction of long paths in  Lemma~\ref{LemRobustness}.}
	
 	\item  A continuum version of the \emph{scale-free percolation} model introduced by Deijfen et al. \cite{DeijfenHofstadHooghiemstra2013, heydenreich2017}, is given by the \emph{product kernel}
 		\[g^{\text{prod}}(s,t)=\beta^{-1}s^\gamma t^\gamma,\]
	see \cite{Deprez2015, Deprez2018} for more details. For this model it is known that  there is no percolation phase transition if $\gamma> \frac12$, but there is one if $\gamma<\frac12$. As the product kernel and the preferential attachment kernel coincide for $\gamma=\frac12$, it follows that the scale-free percolation model has $p_c>0$ at the critical parameter $\gamma=\frac12$ for a general class of profile functions $\rho$. 
For more information how to translate the parameters of that model to our setting see \cite[Table 2]{gracar2019recurrence}. 

 	\item Our result also shows that for profile functions $\rho$ that decay faster than any polynomial, there always exists a subcritical phase. This applies in particular to the Boolean model mentioned above where $\rho$ is the indicator function, see also \cite{gouere2008}.
 \end{enumerate}

\pagebreak[3]

\subsection*{Robustness of age-based preferential attachment networks} 

Let $\G_0$ be the age-dependent random connection model with a vertex at the origin. That is, $\G_0$ is the graph with
\begin{itemize}
\item vertex set obtained from a standard Poisson point process in $\R^d\times (0,1]$ with
an additional point $\0=(0,U)$ placed at the origin with inverse weight, resp.\ birth time $U$, sampled independently from everything else from the uniform distribution on $(0,1]$, 
\item edges laid down independently with connection probabilities given by the preferential attachment kernel { \eqref{PAKernel}}, i.e.\ 	
\[\varphi((x,t),(y,s))=\rho(g^{\text{pa}}(t,s)|x-y|^d).\]
\vspace{-7mm}
\end{itemize}
Theorem~\ref{ThmPercolationPhase}  applies to the graph $\G_0$, which plays a special role as weak local limit in the sense of Benjamini and Schramm~\cite{benjamini2001} of the age-based spatial preferential attachment model, which we now describe.
\medskip

Let $\mathbb{T}^d_a=(-a^{1/d}/2, a^{1/d}/2]^d$ be the $d$-dimensional torus of volume $a$, endowed with the torus metric $d$ defined by
	\[d(x,y)=\min\big\{|x-y+u|: u\in\{-a^{1/d},0,a^{1/d}\}^d\big\}, \text{ for } x,y\in\mathbb{T}_a^d.\]
The \emph{age-based (spatial) preferential attachment model} is a growing sequence of graphs $(\G_t)_{t\geq 0}$ on 
$\mathbb{T}^d_1$ defined as follows: 
\begin{itemize}
\item The graph $\G_t$ at time $t=0$ has neither vertices nor edges. 
	\item Vertices arrive successively after exponential waiting times with parameter one 
	and are placed uniformly on $\mathbb{T}^d_1$. We denote a vertex created at time $s$ and placed in 
	$y\in\mathbb{T}^d_1$ by $\y=(y,s)$.  
	\item Given the graph $\G_{t-}$, a vertex $\x=(x,t)$, born at time $t$ and placed at $x$ is connected by an edge to each existing vertex $\y=(y,s)$ independently with conditional probability
	\begin{equation}
		\rho\left(\mbox{${\frac{1}{\beta}}\frac{t \, d(x,y)^d}{\left(t/s\right)^\gamma}$}\right). \label{PAProb}
	\end{equation}
\end{itemize}
\vspace{-5mm}
Note that the connection probability has the same form as the previously defined connection function $\varphi$, where the Euclidean distance is replaced by the torus distance.%
\medskip%

We say that such a network $(\G_t)_{t\geq 0}$ has a \emph{giant component} if its largest connected component is asymptotically of linear size. More precisely, let $|\C_t|$ be the size of the largest component in $\G_t$ . Then, $(\G_t)_{t\geq 0}$ has a giant component if
 	\[\lim_{\varepsilon\downarrow 0}\limsup_{t\to\infty}\P\big\{\tfrac1t {|\C_t|}<\varepsilon\big\}=0.\]
We say $(\G_t)_{t\geq 0}$ is \emph{robust} if the percolated sequence $(\G^p_t)_{t \geq0}$ has a giant component for every retention parameter $p>0$. Otherwise we say the network is \emph{non-robust}. The idea of this definition is that a random attack cannot significantly affect the connectivity of a robust network.\pagebreak[3]
\medskip

\begin{theorem} \label{ThmRobustness}
Suppose $\rho$ satisfies \eqref{RegularVarying}  for some $\delta>1$
and  $(\G_t)_{t\geq 0}$ is the age-based preferential attachment network with  parameters $\beta>0$ and
$0<\gamma<1$. Then the network  $(\G_t)_{t\geq 0}$ is robust if $\gamma>\frac{\delta}{\delta+1}$, but
non-robust if $\gamma<\frac{\delta}{\delta+1}$.
\end{theorem}
\medskip
\pagebreak[3]

\textbf{Remarks:} 
\begin{enumerate}[(i)]
\item As $\tau=1+\frac1\gamma$ the condition $\gamma<\frac{\delta}{\delta+1}$ is equivalent to
	$\tau>2+\frac1{\delta}$. Hence the qualitative change in the behaviour does not occur when $\tau$
	passes the critical value $3$ as in the classical scale-free network models without spatial correlations, 
	but when it passes a strictly smaller value. This shows
	the significant effect of clustering on the network topology.
\item Replacing $(t/s)^\gamma$ in \eqref{PAProb} by $f(\text{indegree of } (y,s) \text{ in }\G_{t-})$, for some increasing function $f$, we obtain  the spatial preferential attachment model of~\cite{JacobMoerters2015}. If $f$ is a function of asymptotic linear slope $\gamma$, then $(t/s)^\gamma$ is the asymptotic expected degree at time $t$ of a vertex born at time $s$. The age-based preferential attachment model is therefore a simplification and approximation of the spatial preferential attachment model showing very similar behaviour. 
In \cite{JacobMoerters2017} Jacob and M\"orters show that the spatial preferential attachment model is robust for \smash{$\gamma>\frac\delta{\delta+1}$ }but it remains an open problem to show non-robustness for
\smash{$\gamma<\frac\delta{\delta+1}$} for this model. Theorem~\ref{ThmRobustness} is a strong indication that this is the case.
\end{enumerate}

\pagebreak[3]
The remainder of the paper is organized as follows. In Section~\ref{SecPhaseTrans} we prove
existence of a percolation phase transition claimed in Theorem~\ref{ThmPercolationPhase}(a). This proof is based on a novel path decomposition argument and constitutes the main new contribution of this paper.
The remaining proofs are similar to the corresponding arguments for spatial preferential attachment 
in \cite{JacobMoerters2015, JacobMoerters2017}, namely the absence of a phase transition in Theorem~\ref{ThmPercolationPhase}(b) in Section~\ref{SecSupercrit} and the proof of Theorem~\ref{ThmRobustness}, in Section~\ref{SecRobustness}, and will only be  sketched.
Some technical calculations are deferred to the appendix.


\section{Existence of a subcritical phase} \label{SecPhaseTrans}

In this section, we prove Theorem 1.1(a). This proof works for all kernels $g$ which are bounded from below by a constant multiple of
the preferential attachment kernel~$g^{\rm pa}$, similarly the proof of Theorem 1.1 (b) given in Section 3 works for all kernels bounded 
from above by a multiple of the min kernel~$g^{\rm min}$. 

\subsection*{Graphical construction of the  model} We explicitly construct the weight-dependent random connection model on a given countable set $\mathcal{Y}\subset \mathbb{R}^d\times (0,1]$. Let $E(\mathcal{Y})=\set{\set{\x,\y}:\x,\y\in\mathcal{Y}}$ be the set of potential edges and $\mathcal{V}=(\mathcal{V}(e))_{e\in E(\mathcal{Y})}$ a sequence in $[0,1]$ indexed by the potential edges. We then construct the graph $\mathcal{G}_{\varphi}(\mathcal{Y},\mathcal{V})$ through its vertex set $\mathcal{Y}$ and edge set
$$\set[\big]{\set{\x,\y}: \mathcal{V}(\{\x,\y\})\leq\varphi(\x,\y)}.$$
Let $\X$ be a Poisson point process on $\mathbb{R}^d\times (0,1]$ and $\mathcal{U}=(\mathcal{U}(e))_{e\in E(\X)}$ an independent  sequence of in $(0,1)$ uniformly distributed random variables, then $\G=\mathcal{G}_{\varphi}(\X,\mathcal{U})$ is the  weight-dependent random connection model with connectivity function $\varphi$. If
$p\in(0,1]$ then $\G^p=\mathcal{G}_{p\varphi}(\X,\mathcal{U})$ is the percolated model with retention parameter~$p$.%
\pagebreak[3]%

Add to $\X$ a vertex $\mathbf{0}=(0,U)$, placed at the origin with inverse weight $U$ distributed uniformly on $(0,1)$, independent of everything else, and denote the resulting point process by $\X_0$. 
Insert further independent uniformly distributed random variables $(U_{\{\0,\x\}})_{\x\in\X}$ into the family $\mathcal{U}$ and denote the result by~$\mathcal{U}_0$ and the underlying probability measure by $\P_0$. The graph 
$\G_0^p=\mathcal G_{p \varphi}(\X_0,\mathcal{U}_0)$ is the Palm version of $\G^p$, we denote its law by
$\P^p_0$ and expectation by $\E_0^p$. Writing $\P_{(x,t)}^p$ for the law of $\G^p$ conditioned on the event that $(x,t)$ is a vertex of $\G^p$, we have 
\smash{$\P^p_0=\P^p_{(0,u)} \d u$}.  Roughly speaking, this construction ensures that $\0$ is a 
typical vertex in $\G_0^p$. \\[-8mm]

\subsection*{Percolation} 

For two given points $\x$ and $\y$, we denote by $\{\x\sim\y\}$ the event  that $\x$ and $\y$ are connected by an edge in 
$\G_0^p$. We define $\set{\0\leftrightarrow\infty}$ as the event that $\0=\x_0$ is the starting point of an \emph{infinite self-avoiding path} $(\x_0, \x_1,\x_2,\dots)$ in $\G_0^p$. That is, $\x_i\in\X$ for all~$i$, $\x_i\neq\x_j $ for all $i\neq j$, and $\x_i\sim\x_{i+1}$ for all $i\geq 0$. If $\{\0\leftrightarrow\infty\}$ occurs, we say that $\G_0^p$ percolates. We denote the percolation probability by
	\begin{equation}
		\theta(p)=\P_0^p\left\{\0\leftrightarrow\infty\right\} = \int_0^1 \d u \ \P^p_{(0,u)}\{(0,u)\leftrightarrow\infty\},\label{percolationProb}
	\end{equation}
which can be interpreted as the probability that a typical vertex belongs to the infinite cluster. 
We define the critical percolation parameter as
	\begin{equation}
		p_c:=\inf\left\{p\in(0,1]: \theta(p)>0\right\}.
	\end{equation}


\subsection*{Existence of a subcritical phase: Case $\gamma<\frac12$.} \label{SecNonPercolativeLeq}


We fix $\delta>1,\beta>0$ and $\gamma<\frac\delta{\delta+1}$. Since $g^\text{pa}\leq g^\text{min}\leq { 2^d g^\text{sum}}$, 
we have
	\begin{align*}
	\P_0\{\0\leftrightarrow\infty \text{ in }\G^p_0(\rho\circ g^\text{pa})\} & \geq \P_0\{\0\leftrightarrow\infty \text{ in }\G^p_0(\rho\circ g^\text{min})\} 
	\\ & \geq \P_0\{\0\leftrightarrow\infty \text{ in }\G^{2^dp}_0(\tilde\rho\circ g^\text{sum})\} \label{eqCoupling}
	\end{align*}%
	\pagebreak[3]%
for $\tilde\rho(x)= \frac{1}{2^d} \rho(2^dx)$ by a simple coupling argument. Thus, we focus on the preferential attachment kernel and show that we can choose a $p>0$ such that $\theta(p)=0$. Consequently, we work in the following exclusively in the age-dependent random connection model, and we therefore use the corresponding terminology. For a vertex $\x=(x,t)$ we refer to $t$ as the \emph{birth time} of $\x$ and, for another vertex $\y=(y,s)$ with $s<t$, we say $\y$ is \emph{older} than $\x$. We also say $\y$ is born before $\x$, or before $t$. 

We use a \emph{first moment method} approach for the number of paths of length $n$. We start with $\gamma<\frac12$
and explicitly calculate the expected number of such paths. This turns out to be  independent of the spatial geometry of the model and therefore cannot be used to prove the statement for \smash{$\frac12\leq\gamma<\frac\delta{\delta+1}$.}
We denote by $\mathbf{E}$ the expectation of a Poisson point process on $\R^d\times(0,1]$ of unit intensity, by $\P^p_{\X}$ the law of $\G^p$ conditioned on the whole vertex set $\X$
and by $\P^p_{\x_1,\dots,\x_n}$ the law of $\G^p$ conditioned on the event that $\x_1,\dots,\x_n$ are points of the vertex set. 
\bigskip

\begin{lemma}\label{LemPercolationLeq} {If $0<\gamma< \frac12$, then $\theta(p)=0$ for all $p<\tfrac{1-2\gamma}{4\beta}$ or, equivalently, $p_c\geq \tfrac{1-2\gamma}{4\beta}$}.
\end{lemma}

\begin{proof}
 We set $\0=\x_0=(0, t_0)$ and get
	\begin{align*}
		\theta(p) & =\lim_{n\to\infty} \P^p_0\{\exists \text{ a path of length }n\text{ starting in }\x_0\}\\
					   &\leq \lim_{n\to\infty}\int_0^1 \d t_0 \ \mathbf{E}\bigg[\underset{x_i\neq x_j\forall i\neq j}{\sum_{\x_1,\dots,\x_n\in\X}}\P^p_{\X\cup\{(0,t_0)\}}\Big(\bigcap_{j=1}^n\{\x_j\sim\x_{j-1}\}\Big)\bigg]. 
	\end{align*}
The inner probability is a measurable function of the Poisson process and the points
$\x_1,\dots,\x_n$ and by Mecke's equation \cite[Theorem 4.4]{LastPenrose2017} we get,
with $\eta$ denoting an independent copy of $\X$, 
	\begin{align*}
	& \int_0^1 \d t_0 \ \int\limits_{(\mathbb{R}^d\times(0,1])^n}\bigotimes_{j=1}^n\d\x_j \ \mathbf{E}\left[\P^p_{\eta\cup\{(0,t_0),\x_1,\dots,\x_n\}}\left(\bigcap_{j=1}^n\{\x_{j-1}\sim\x_j\}\right)\right] \\
	&\quad= \int_0^1 \d t_0 \ \int\limits_{(\mathbb{R}^d\times(0,1])^n}\bigotimes_{j=1}^n\d\x_j \ 
	\P^p_{\x_0,\ldots,\x_n}\left(\bigcap_{j=1}^n\{\x_{j-1}\sim\x_j\}\right).
	\end{align*}
Given the vertices, edges are drawn independently so we get by writing $\x_j=(x_j,t_j)$ for all \(j\in\{1,\dots,n\}\) that the previous expression equals
	\begin{align*}
		& \int_0^1 \d t_0 \ \int\limits_{(\mathbb{R}^d\times(0,1])^n}\bigotimes_{j=1}^n\d(x_j,t_j) \bigg(\prod_{j=1}^n p\rho\big({ g^{\text{pa}}(t_{j-1}, t_j)}|x_j-x_{j-1}|^d\big)\bigg) \\
	&\quad= p^n\beta^n\int_0^1 \d t_0 \int_0^1\d t_1 \dots\int_0^1 \d t_n \bigg(\prod_{j=1}^n (t_j\wedge t_{j-1})^{-\gamma}(t_j\vee t_{j-1})^{\gamma-1}\bigg),
	\end{align*}
where we used the normalization condition \eqref{IntegrabilityCond}. Since $\gamma <\frac12$, Lemma~17 of \cite{JacobMoerters2017} states that
	\[\int_0^1 \d t_0 \int_0^1\d t_1 \dots\int_0^1 \d t_n \bigg(\prod_{j=1}^n (t_j\wedge t_{j-1})^{-\gamma}(t_j\vee t_{j-1})^{\gamma-1}\bigg)\leq \bigg(\frac{1}{1+\alpha-\gamma}-\frac{1}{\alpha+\gamma}\bigg)^n,\]
for $\alpha\in(\gamma-1,-\gamma)$. The minimum of the right-hand side over this non-empty interval
equals $\frac4{1-2\gamma}$ and thus, setting $p<\frac{1-2\gamma}{4\beta}$ we achieve
	\[\theta(p)\leq \lim_{n\to\infty}\big(\tfrac{4p\beta}{1-2\gamma}\big)^n=0.\]
	\ \\[-12mm]
\end{proof}
\pagebreak[3]

\subsection*{Existence of a subcritical phase: Case $\gamma\geq\frac12$.}  \label{SecNonPercolativGeq}
We now turn to the more interesting case when $\gamma\in[\frac12,\frac\delta{\delta+1})$ where we have to use the spatial properties of our model in order to prove our claim. Intuitively, as ``powerful'' vertices are typically far apart from each other, in order to create an infinite path in this spatial network one has to use long edges often enough to reach them. Therefore, where the long edges are used is the crucial and most interesting part of a path. 
On the other hand $\G$ is locally dense. Therefore, considering paths that stay for a long time in a neighbourhood of a vertex before using long edges greatly increases the number of possible paths we can construct. For $\gamma<\frac12$, the degrees of typical vertices are small enough so that the number of possible paths does not increase too much. This is not true anymore for {$\gamma>\frac12$} where the degree distribution has an infinite second moment. Thus, it becomes difficult to bound the probability of the existence of an arbitrary path of length $n$. In order to prove the existence of a subcritical phase, we start by explaining how to limit our counting to paths that are not stuck in local clusters. Then, we define what we call the \emph{skeleton} of a path, which will help with counting the valid paths. As we will see, the skeleton is a collection of key vertices from a path ordered in a specific birth-time structure. In the end, we will use these paths to complete the proof of Theorem \ref{ThmPercolationPhase}(a). 
\pagebreak[3]

\paragraph{Shortcut-free paths} \label{SubsubShortCut}
Let $P=(v_0,v_1,v_2,\dots)$ be a path in some graph $G$. We say $(v_i,v_j)$ is a \emph{shortcut} in $P$ if $j>i+1$ and $v_i$ and $v_j$ are connected by an edge in \(G\). If $P$ does not contain any shortcut, we say $P$ is \emph{shortcut-free}. If $G$ is locally finite, i.e.\ all vertices of $G$ are of finite degree, then there exists an infinite path if and only if there exists one that is also shortcut-free. To see how an infinite path $P=(v_0,v_1,v_2\dots)$  in $G$ can be made shortcut-free define $i_0=\max\{i\geq 1: v_i\sim v_0\}$. If $i_0=1$, then $v_1$ is the only neighbour $v_0$ has in~$P$. If $i_0\geq 2$, then $(v_0,v_{i_0})$ is a shortcut in $P$ so we remove the vertices $v_1,\dots,v_{i_0-1}$ from $P$. We have thus removed all shortcuts starting from $v_0$ and since $v_0\sim v_{i_0}$ the new $P$ is still a path. We define analogously $i_k=\max\{i>i_{k-1}: v_i\sim v_{i_{k-1}}\}$ for every $k\geq 1$ and remove the intermediate vertices as needed. The resulting path $(v_0,v_{i_0},v_{i_1},\dots)$ is then still infinite but also shortcut-free. 

\paragraph{Skeleton of a path}\label{SubSubSkeleton}
Let $P=((v_0,t_0),(v_1,t_1),\dots,(v_n,t_n))$ be a path of length $n$ in some graph $G$ where every vertex $v_i$ carries a distinct birth time $t_i$. Then, precisely one of the vertices in $P$ is the oldest; let $k_\text{min}=\{k\in\{0,\dots,n\}:t_k<t_j, \ \forall j\neq k\}$ be its index. Starting from $(v_0,t_0)$, we now choose the first vertex of the path that has birth time smaller than $t_0$ and call it $(v_{i_1},t_{i_1})$. Continuing from this vertex, we choose the next vertex of the path that is older still, call it $(v_{i_2},t_{i_2})$ and continue analogously until we reach the oldest vertex \((v_{k_\text{min}},t_{k_{\text{min}}})\). We then repeat the same procedure starting from the end vertex $(v_n,t_n)$ and going backwards across the indices. The union of the two subset of vertices is what we call the \emph{skeleton} of the path $P$. More precisely, for every path $P=((v_0,t_0),\dots,(v_n,t_n))$, there exists unique $0\leq k\leq n$ and $k\leq m\leq n$ as well as a set of indices $\{i_0,i_1,\dots,i_{k-1}, i_k, i_{k+1},\dots, i_{m}\}$ such that 
\begin{align*}
	& i_0=0, i_k=k_{\text{min}}, \text { and } i_m=n \text{ as well as } \\
	& t_{i_{\ell-1}}>t_{i_\ell} \text{ and } t_{i}>t_{i_{\ell-1}}, \ \forall i_{\ell-1}<i<i_\ell, \text{ for } \ell=1,\dots, k \text{ and } \\
	& t_{i_{\ell-1}}<t_{i_\ell} \text{ and }t_i>t_{i_\ell}, \ \forall i_{\ell-1}<i<i_\ell, \text{ for } \ell=k+1,\dots m.
\end{align*}
The \emph{skeleton of $P$} is then given by $((v_{i_j},t_{i_j}))_{j=0,\dots,m}$. We say it is of length $m$ and has its minimum at $k$. \pagebreak[3]

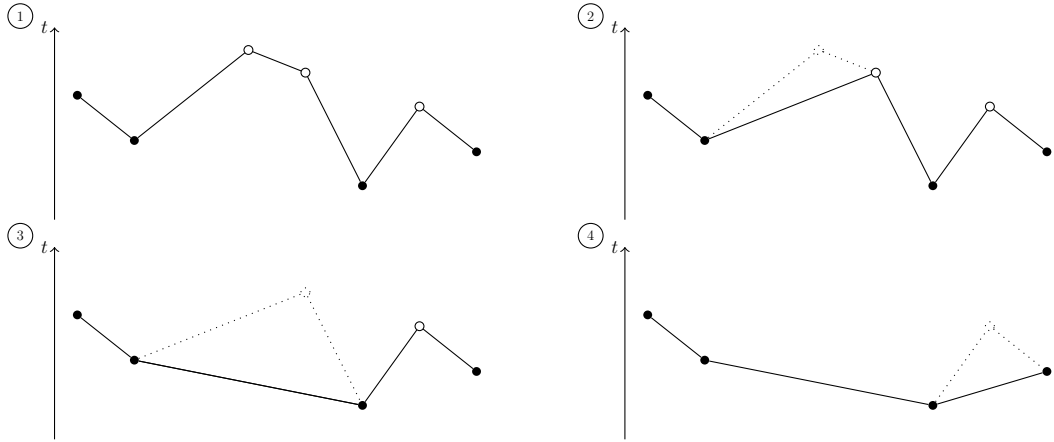
\begin{figure}
\begin{center}
			\begin{tikzpicture}[scale=0.3, every node/.style={scale=0.3}]
				\node (Z) at (-2.5,8.5)[circle, draw,scale=1.5]{1};
				\draw[->] (-1,-0.5) -- (-1, 8)
					node[left,scale=2] {$t$};
				\node (A) at (0,5)[circle, fill=black, label ={}]{};	
    			\node (B) at (2.5,3)[circle, fill = black, label={}] {};
    			\node (D) at (7.5,7)[circle, draw, label = {}] {};
    			\node (E) at (10, 6)[circle, draw, label={}] {};
    			\node (F) at (12.5,1) [circle, fill= black, label={} ] {};
    			\node (G) at (15,4.5) [circle, draw, label={}]{};
    			\node (H) at (17.5,2.5)[circle, fill=black, label={}]{};
    			\draw (A) to (B);
				\draw (B) to (D);
				\draw (D) to (E);
				\draw (E) to (F);
				\draw (F) to (G);
				\draw (G) to (H);
			\end{tikzpicture}
			\hspace{1 cm}
			\begin{tikzpicture}[scale=0.3, every node/.style={scale=0.3}]
				\node (Z) at (-2.5,8.5)[circle, draw,scale=1.5]{2};
				\draw[->] (-1,-0.5) -- (-1, 8)
					node[left,scale=2] {$t$};
				\node (A) at (0,5)[circle, fill=black, label ={}]{};
    			\node (B) at (2.5,3)[circle, fill = black, label={}] {};
    			\node (D) at (7.5,7)[circle, draw, label = {}, dotted] {};
    			\node (E) at (10, 6)[circle, draw, label={}] {};
    			\node (F) at (12.5,1) [circle, fill= black, label={} ] {};
    			\node (G) at (15,4.5) [circle, draw, label={}]{};
    			\node (H) at (17.5,2.5)[circle, fill=black, label={}]{};
    			\draw (A) to (B);
				\draw[dotted] (B) to (D);
				\draw[dotted] (D) to (E);
				\draw (B) to (E);
				\draw (E) to (F);
				\draw[] (F) to (G);
				\draw[] (G) to (H);
			\end{tikzpicture}
			\begin{tikzpicture}[scale=0.3, every node/.style={scale=0.3}]
				\node (Z) at (-2.5,8.5)[circle, draw,scale=1.5]{3};
				\draw[->] (-1,-0.5) -- (-1, 8)
					node[left,scale=2] {$t$};
				\node (A) at (0,5)[circle, fill=black, label ={}]{};
    			\node (B) at (2.5,3)[circle, fill = black, label={}] {};
    			\node (D) at (7.5,7)[label = {}] {};
    			\node (E) at (10, 6)[circle, draw, label={}, dotted] {};
    			\node (F) at (12.5,1) [circle, fill= black, label={} ] {};
    			\node (G) at (15,4.5) [circle, draw, label ={}]{};
    			\node (H) at (17.5,2.5)[circle, fill=black, label={}]{};
    			\draw (A) to (B);
				\draw (B) to (F);
				\draw[dotted] (B) to (E);
				\draw[dotted] (E) to (F);
				\draw (B) to (F);
				\draw[] (F) to (G);
				\draw[] (G) to (H);
			\end{tikzpicture}
			\hspace{1 cm}
			\begin{tikzpicture}[scale=0.3, every node/.style={scale=0.3}]
				\node (Z) at (-2.5,8.5)[circle, draw,scale=1.5]{4};
				\draw[->] (-1,-0.5) -- (-1, 8)
					node[left,scale=2] {$t$};
				\node (A) at (0,5)[circle, fill=black, label ={}]{};
    			\node (B) at (2.5,3)[circle, fill = black, label={}] {};
    			\node (D) at (7.5,7)[label = {}] {};
    			\node (E) at (10, 6)[label={}] {};
    			\node (F) at (12.5,1) [circle, fill= black, label={} ] {};
    			\node (G) at (15,4.5) [circle, draw, dotted, label={}]{};
    			\node (H) at (17.5,2.5)[circle, fill=black, label={}]{};
    			\draw (A) to (B);
				\draw[dotted] (F) to (G);
				\draw[dotted] (G) to (H);
				\draw (B) to (F);
				\draw (F) to (H);
			\end{tikzpicture}
			\caption{A path where a vertex's birth time is denoted on the $t$-axis. The vertices of the skeleton are in black. We successively remove all local maxima, starting with the youngest, and replace them by direct edges until the path, only containing the skeleton vertices, is left.}
	\label{FigSkeleton}
\end{center}
\end{figure}

We now give an alternative construction of the skeleton of \(P\), which we call the \emph{local maxima construction}. A vertex $(v_i,t_i)\in P\backslash\{(v_0,t_0),(v_n,t_n)\}$
is called  a \emph{local maximum} if $t_i>t_{i-1}$ and $t_i>t_{i+1}$. 
We  successively remove all local maxima from $P$ as follows: First, take the local maximum in $P$ with the greatest birth time, remove it from $P$ and connect its former neighbours by a direct edge. In the resulting path, we take the local maximum of greatest birth time and remove it, repeating until there is no local maximum left, see Figure~\ref{FigSkeleton}. Therefore, the final path is decreasing in birth times of its vertices until the oldest vertex is reached, and only increasing in birth times afterwards. 
Hence, it  is the uniquely determined skeleton of the path. 
Note that the skeleton is not necessarily an actual path of the graph. 
{Actually, the skeleton of a shortcut-free path is not itself a path unless the path is its own skeleton.}%

\paragraph{Graph surgery}
In order to  bound the probability of existence of an infinite self-avoiding path in $\G^p_0$ starting in the origin
we increase the number of short edges in $\G^p_0$, which then allows us to make better use of the shortcut-free condition. We choose $\varepsilon>0$ such that $$\tilde{\delta}:=\delta-\varepsilon>\frac\gamma{1-\gamma}.$$ 
This is equivalent to $\gamma<\frac{\tilde{\delta}}{\tilde{\delta}+1}$. As $\rho$ is regularly varying 
and bounded there exists $A>1$ such that 
	\[\rho(x)\leq A x^{-\tilde{\delta}} \quad \mbox{ for all $x>0$,} \]
by the Potter bound \cite[Theorem 1.5.6]{Bingham1987}. We define 
	\[\tilde{\rho}(x)=\mathbbm{1}_{[0,(pA)^{1/{\tilde\delta}}]}(x)+pA x^{-\tilde{\delta}}\mathbbm{1}_{((pA)^{1/{\tilde\delta}},\infty)}(x).\] 
{Note that $p$ enters the definition of 
$\tilde{\rho}$
at two places. Namely, it determines the range where edges are put deterministically and 
also scales the profile function.}
We now choose $\tilde{\rho}$ as a profile function together with the preferential attachment kernel \eqref{PAKernel} and construct $\mathcal{G}_{\tilde{\varphi}}(\X_0,\mathcal{U}_0)$ where 
	\[\tilde{\varphi}((x,t),(y,s))= 
	\tilde{\rho}\big({g^{\text{pa}}(t,s)}|x-y|^d\big).\]
In other words, we connect two given vertices $(x,t)$ and $(y,s)$ with probability
	\begin{align*}
		\begin{cases} 		
 			1, & \text{ if }|x-y|^d\leq (pA)^{\frac1{\tilde\delta}}{\gp(t,s)^{-1}} \\
 			p A\, \big({\gp(t,s)}|x-y|^d\big)^{-\tilde{\delta}}, & \text{ otherwise}.
		\end{cases}
	\end{align*} 			
Note that  in general  \(\tilde\rho\) does not satisfy the normalization condition \eqref{IntegrabilityCond}. However, $\tilde{\rho}$ is still integrable and therefore the resulting graph $\mathcal{G}_{\tilde{\varphi}}(\X_0,\mathcal{U}_0)$ is still locally finite with unchanged power law and shows qualitatively the same behaviour. 
Since 
$p\rho\leq \tilde{\rho}$, it follows by a simple coupling argument that
	\[\theta(p)\leq \P_0\{\0\leftrightarrow\infty \text{ in }\mathcal{G}_{\tilde{\varphi}}(\X_0, \mathcal{U}_0)\}.\]
Due to the above it is no loss of generality to consider the unpercolated graph $\G$, resp.~$\G_0$,
where the profile function
 $\rho$ is of the form
	\begin{equation}
		\rho(x) = 1 \wedge (pAx^{-\delta}),\label{defRho}
	\end{equation}
which is what we do from now on. Note that we can no longer assume that~\eqref{IntegrabilityCond} holds, 
instead we have 
	\begin{align}\label{def_irho}
	I_{\rho} & :=\int_{\R^d} \rho(|x|^d) \,  \d x =  (pA)^{1/\delta} \big(J(d)\tfrac{\delta}{d(\delta-1)}\big)
	\end{align}
where $J(d)=\prod_{j=0}^{d-2}\int_0^\pi \sin^j(\alpha_j)\d \alpha_j$ is the Jacobian of the $d$-dimensional sphere coordinates. We look at the probability that a shortcut-free path $P=((x_1,t_1),(x_2,t_2),\dots)$ exists in $\G$. 
By choice of $\rho$, such a path satisfies
$$|x_i-x_j|^d> (pA)^{\frac1{\delta}} {\gp(t_i,t_j)^{-1}},
\quad \mbox{ for all $|i-j|\geq 2$.}$$
\vspace{-8mm}

\paragraph{Strategy of the proof}
{ To build a long path, one needs to use old vertices. Every path is divided into a skeleton, which encodes how it moves to increasingly old vertices, and subpaths
connecting consecutive points of the skeleton by any number of younger vertices, which we call \emph{connectors}. We encode a characteristic feature of such a subpath by an unlabelled  binary tree using the local maxima construction. We show that whenever $\gamma< \delta/(\delta+1)$ the expected number of shortcut-free subpaths with a given tree  
of size $k$ is bounded by $(K I_{\rho})^k$ times the probability that the two extremal vertices are connected by an edge, for some constant $K>1$. Combining this estimate with the BK-inequality allows us to bound the probability of existence of a path with a given skeleton in terms of the probability that this skeleton is a path. The probability of existence of paths of the latter type can be estimated by a truncated first moment method with the truncation applied to the birth time of the oldest point on the skeleton. We therefore obtain that the probability of existence of a shortcut-free path of length $n$ starting in $\0$ is bounded from above by $(K I_\rho)^n$  and hence \[\theta(p)\leq\lim_{n\to\infty}(K I_\rho)^n=0\]
for $p>0$ small enough to ensure $I_\rho<1/K$.} 
	
\paragraph{Connecting two old vertices}
Let \(P\) be a path of length \(k\) that can be reduced to a skeleton with two vertices \(\x\) and~\(\y\). Let \(\y_0,\dots,\y_k\) be the vertices of \(P\), ordered by age from oldest to youngest. We assume without loss of generality that \(\x\) is younger than \(\y\) and therefore \(\x=\y_1\) and \(\y=\y_0\). We denote by $\mathscr{T}_{k-1}$ the set of all binary 
trees\footnote{Here, a binary tree is a rooted tree in which every vertex can have either (i) no child, (ii) a left child (iii) a right child, or (iv) a left and a right child.} 
with fixed vertex set \(\{\y_2,\dots,\y_k\}\)
such that every child has birth time greater than its parent.  With the path \(P\)
we associate a tree in $\mathscr{T}_{k-1}$ as follows, see Figure~\ref{FigPathToTree}.
\begin{description}
	\item[Step one:]  $\y_2$ is the root of the tree. 
\item[Step two:]  Suppose the tree with vertices $\y_2,\dots,\y_{i-1}$ is constructed. Attach
$\y_i$ {as a new leaf of} the tree.  To find {the place to attach the leaf} start at  the root and branch at every vertex to the left if
the path $P$ visits $\y_i$ before the vertex and to the right otherwise. If this means going to a place where 
there is no vertex, we attach $\y_i$ there. We continue like this until all $\y_2,\ldots,\y_k$  are attached.
\end{description}


\begin{figure}
\begin{center}
	\begin{minipage}{0.4\textwidth}
	\begin{tikzpicture}[scale=0.7, every node/.style={scale=0.7}]
		\draw[->] (-1,-0.5) -- (-1, 8)
			node[left] {$t$};
    	\node (A) at (0,1)[circle, fill=gray, label={$\y_1$}] {};
    	\node (B) at (2,3)[circle, draw, label={$\y_3$}] {};
    	\node (C) at (7,2) [circle, draw, label={$\y_2$}] {};
    	\node (D) at (8.5,0)[circle, fill=gray, label = $\y_0$] {};
    	\node (E) at (3,4.5)[circle, draw, label={$\y_4$}]{};
    	\node (F) at (5.5, 7)[circle, fill=black, label=$\y_{6}$]{};
    	\node (G) at (1,6)[circle, draw, label={$\y_5$}]{};
		\draw[gray] (A) to (G);
		\draw[gray](G) to (B);
		\draw[gray] (B) to (E);
		\draw[gray] (C) to (D);
		\draw[gray] (E) to (F);
		\draw[gray] (F) to (C);
	\end{tikzpicture}
	\end{minipage}
	\hspace{0.1\textwidth}
	\begin{minipage}{0.45\textwidth}
	\begin{tikzpicture}[scale=0.7, every node/.style={scale=0.7}]
		\node(C) at (0,0)[circle, draw, label={$\y_{2}$}]{};
		\node(B) at (-4,2)[circle, draw, label={$\y_{3}$}]{};
		\node(E) at (-2,4)[circle, draw, label={$\y_{4}$}]{};
		\node(F) at (0,6)[circle, fill=black, label=$\y_{6}$]{};
		\node(G) at (-6,4)[circle, draw, label={$\y_{5}$}]{};
		\draw (C) to (B);
		\draw (B) to (E);
		\draw[thick](E) to (F);
		\draw (B) to (G);
		\draw[dashed] (G) to (-7,5);
		\draw[dashed] (G) to (-5,5);
		\draw[dashed] (E) to (-3,5);
		\draw[dashed] (C) to (1,1);
	\end{tikzpicture}
	\end{minipage}
	\caption{On the left the path $P$ where the $t$-axis denotes the vertices' birth times. The vertices $\y_1$ and $\y_0$, which will not appear in the tree, are in grey. We insert the vertex $\y_6$ at the end of the branch that
	goes left at $\y_2$, right at $\y_3$, and right at $\y_4$.}
\label{FigPathToTree}
\end{center}
\end{figure}
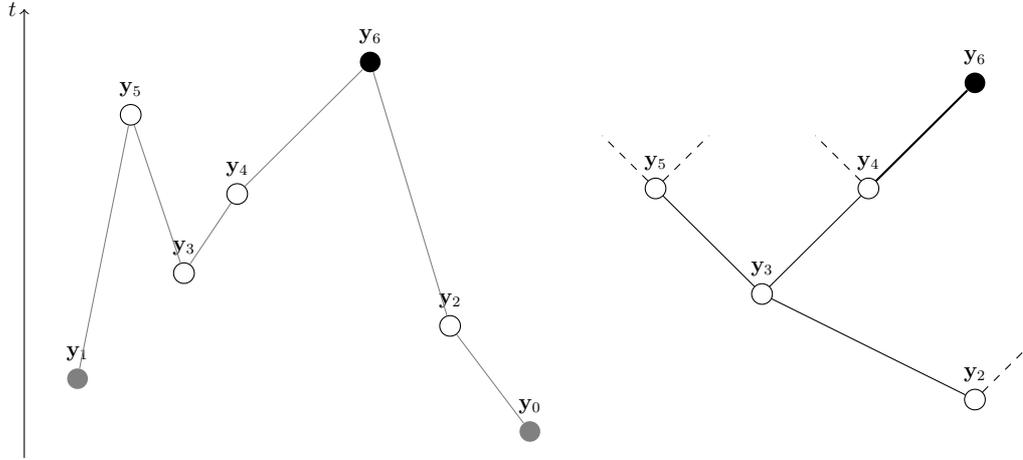

Next, we explain how to construct a path $P$ connecting $\x$ and $\y$ when $T\in\mathscr{T}_{k-1}$ is given, see Figure \ref{FigTreeToPath}. Here, given a path $(v_i)_{i=1}^n$ and any subpath $(v_{j-1},v_j,v_{j+1})$, we call $v_{j-1}$ the \emph{preceding vertex} of $v_j$ and $v_{j+1}$ the \emph{subsequent vertex} of $v_j$.
We explore $T$ using depth-first search and add the vertex currently being explored to the path. Let $P=(\x,\y)$ and let $\u$ be the root of $T$. We define $L=(\u)$ to be the list of vertices to be explored next (in the order as they are in \(L\)). We proceed as follows.
\begin{description}
	\item[Step one:] We insert $\u$ into $P$ as a local maximum between $\x,\y$. As a result $P=(\x,\u,\y)$. 
We remove $\u$ from $L$ and if $\u$ has children in $T$, we add them to $L$, ordered from left to right.

\item[Step two:] While $L$ is not empty, we do the following:
\begin{enumerate}
	\item We take the first vertex in $L$, denote it by $\v$ and remove it from $L$.
	\item If $\v$ has children in $T$, we insert them at the beginning of $L$, ordered from left to right. Having done that, we consider \(\v\) explored.
	\item Let $\w$ be the parent of $\v$ in $T$ and $\{\z_1,\w\}$, $\{\w,\z_2\}$ its incident edges in $P$, where $\z_1$ is the preceding vertex of $\w$ in $P$ and $\z_2$ the subsequent one. If $\v$ is the left child of $\w$, we insert $\v$ as a local maximum between $\z_1$ and $\w$ in $P$ by adding it to the path and replacing the edge $\{\z_1,\w\}$ in $P$ by the two edges $\{\z_1,\v\}$ and $\{\v,\w\}$.
	If $\v$ is a right child, we insert $\v$ as a local maximum between $\w$ and $\z_2$ in an analogous way.
\end{enumerate}
\end{description}
\vspace{-2mm}
\pagebreak[3]

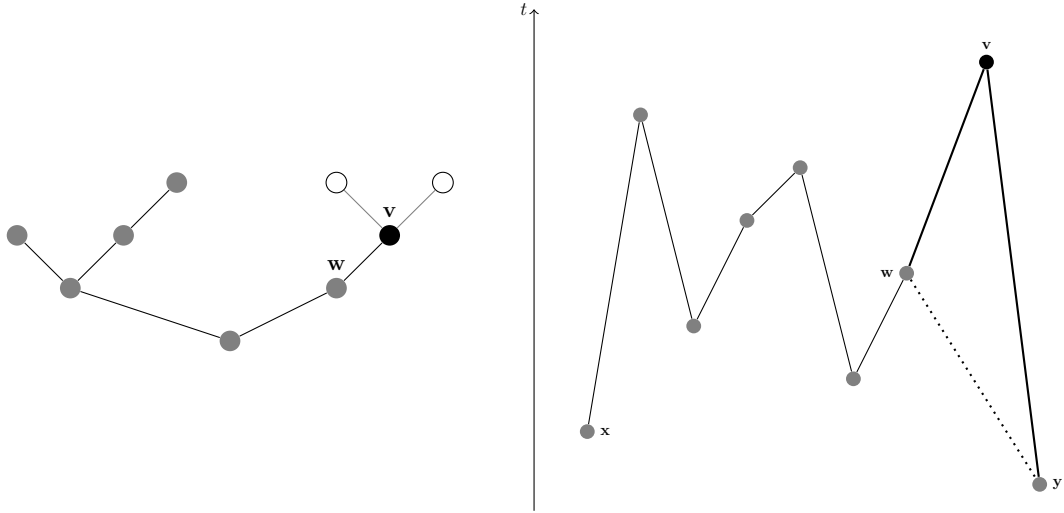
\begin{figure}[]
\begin{center}
	\begin{minipage}{0.4\textwidth}
	\begin{tikzpicture}[scale=0.7, every node/.style={scale=0.7}]
		\node(A) at (0,0)[circle, fill = gray, label={}]{};
		\node(B) at (-3,1)[circle, fill = gray, label={}]{};
		\node(C) at (-2,2)[circle, fill = gray, label={}]{};
		\node(D) at (-1,3)[circle, fill = gray, label={}]{};
		\node(E) at (-4,2)[circle, fill= gray, label={}]{};
		\node(F) at (2,1)[circle,  fill =gray , label = $\w$]{};
		\node(G)at (3, 2)[circle, fill=black, label=$\v$]{};
		\node(H)at (2,3)[circle, draw]{};
		\node(I)at (4,3)[circle, draw]{};
		\draw (A) to (B);
		\draw (B) to (E);
		\draw (B) to (C);
		\draw (C) to (D);
		\draw (A) to (F);
		\draw[] (F) to (G);
		\draw[gray] (G) to (H);
		\draw[gray] (G) to (I);
	\end{tikzpicture}
	\end{minipage}
	\hspace{0.02\textwidth}
	\begin{minipage}{0.5\textwidth}
	\begin{tikzpicture}[scale=0.7, every node/.style={scale=0.5}]
		\draw[->] (-1,-1.5) -- (-1, 8)
			node[left,scale=1.2] {$t$};
		\node(X) at (0,0)[circle, fill = gray, label={right:$\x$}]{};
		\node(Y) at (8.5,-1)[circle, fill = gray, label={right:$\y$}]{};
		\node(E) at (1, 6)[circle, fill=gray, label={}]{};
		\node(B) at (2,2)[circle, fill=gray, label ={}]{};
		\node(C) at (3,4)[circle, fill=gray, label ={}]{};
		\node(D) at (4,5)[circle, fill=gray, label={}]{};
		\node(A) at (5,1)[circle, fill=gray, label={}]{};
		\node(F) at (6,3)[circle, fill= gray, label={left:$\w$}]{};
		\node(G) at (7.5,7)[circle, fill=black, label=$\v$]{};
		\draw (X) to (E);
		\draw (E) to (B);
		\draw (B) to (C);
		\draw (C) to (D);
		\draw (D) to (A);
		\draw[] (A) to (F);
		\draw[thick, dotted] (F) to (Y);
		\draw[thick] (F) to (G);
		\draw[thick] (G) to (Y);
	\end{tikzpicture}
	\end{minipage}
	\caption{On the left the binary tree $T$. The grey vertices are already explored by depth-first search. The black vertex $\v$ is the vertex currently being explored. The white vertices have not been discovered yet. On the right, the path $P$ corresponding to the already explored tree. The $t$-axis denotes the vertices' birth times. Start and end vertex, $\x$ and $\y$, do not appear in the tree. 
	Since $\v$ is the right child of $\w$, we insert $\v$ as a local maximum between $\w$ and $\y$ in the path $P$.}
\label{FigTreeToPath}
\end{center}
\end{figure}

\bigskip
It is clear that for given $\y_0,\dots,\y_k$  the two procedures establish a bijection between the paths with vertices  $\y_0,\dots,\y_k$  that can be reduced to a skeleton with two vertices \(\y_0\) and \(\y_1\) on the one hand, and the trees $T\in\mathscr{T}_{k-1}$ on the other hand. {Removing the labels from a tree in $\mathscr{T}_{k}$ yields a binary tree, which encodes important structural information about the path}.

{The following lemma shows that, if $\gamma<\delta/(\delta+1)$, the probability of two vertices being connected through a single connector is bounded 
by a small multiple of the probability that there exists a direct edge between them. }
\pagebreak[3]

{
For two given vertices $\x$ and $\y$, we denote by \smash{$\{\x\xleftrightarrow[\x,\y]{2}\y\}$} the event that $\x$ and $\y$ are connected by a path of length two where the connector is younger than both of them.\medskip}
 
 \pagebreak[3]
 
\begin{lemma}\label{LemTwoConnection}
	Let $\gamma\in(0,\frac{\delta}{\delta+1})$. Let $\x=(x,t)$ and $\y=(y,s)$ be two given vertices satisfying $|x-y|^d\geq (pA)^{1/\delta} {\gp(t,s)^{-1}}$. Then
		\[\P^p_{\x,\y}\{\x\xleftrightarrow[\x,\y]{2}\y\}\leq \int\limits_{\R^d\times ((t\vee s),1]}\d \mathbf{z} \ \P^p_{\x,\z}\{\x\sim\mathbf{z}\}\P^p_{\y,\z}\{\mathbf{z}\sim\y\}\leq I_{\rho} \, C_1\P^p_{\x,\y}\{\x\sim\y\},\]
	where $C_1= \frac{\beta 2^{d\delta+1}}{\delta(1-\gamma)-\gamma}$.
\end{lemma}
\begin{proof}
Without loss of generality let $t>s$ {in which case $\gp(t,s)=\beta^{-1}s^\gamma t^{1-\gamma}$.} Recall that 
\smash{$\{\x\xleftrightarrow[\x,\y]{2}\y\}$} is the event that $\x$ and $\y$ share a common neighbour that is born after both of them. Such neighbours form a Poisson point process on $\R^d\times(t,1]$ with intensity measure
	\[\rho(\beta^{-1}t^{\gamma}u^{1-\gamma}|x-z|^d)\rho(\beta^{-1}s^\gamma u^{1-\gamma}|z-y|^d)\, \d z \, \d u,\]
see \cite{GracarEtAl2019}, from which the first inequality follows. For the second inequality, we have
 \begin{align*}
 	& \int_t^1 \d u\int_{\R^d}\d z \ \rho(\beta^{-1}t^\gamma u^{1-\gamma}|x-z|^d)\rho(\beta^{-1}s^{\gamma}u^{1-\gamma}|z-y|^d) \\
 	& \leq \int_t^1\d u \Big[\int_{\R^d}\d z \ \rho(\beta^{-1}t^\gamma u^{1-\gamma}|x-z|^d)\rho\left((2^d\beta)^{-1}s^{\gamma}u^{1-\gamma}|x-y|^d\right) \\
 	& \qquad + \int_{\R^d}\d z \ \rho\left((2^d\beta)^{-1}t^\gamma u^{1-\gamma}|x-y|^d\right)\rho(\beta^{-1}s^{\gamma}u^{1-\gamma}|z-y|^d)\Big].
 \end{align*}%
Here, the inequality holds as for all $z\in\mathbb{R}^d$ either $|x-z|\geq \frac{1}{2}|x-y|$ or $|y-z|\geq \frac{1}{2}|x-y|$, 
and $\rho$ is non-increasing. For the first integral, a change of variables leads to
	\begin{align*}
	\int_t^1\d u \ \beta t^{-\gamma}u^{\gamma-1}\rho\left((2^d\beta)^{-1}s^{\gamma}u^{1-\gamma}|x-y|^d\right)I_{\rho}.
	\end{align*}
As $\rho(x)=1\wedge (pAx^{-\delta})$ 
this can be further bound by 
	\begin{align*}
	 pA 2^{d\delta}\beta^{1+\delta}I_{\rho} \int_t^1\d u\ s^{-\gamma\delta}t^{-\gamma}
	 & |x-y|^{-d\delta}u^{-\delta(1-\gamma)+\gamma-1}\\
	 & \leq pA 2^{d\delta} I_{\rho} \frac{\beta^{\delta+1}}{\delta(1-\gamma)-\gamma}(s^\gamma t^{1-\gamma}|x-y|^d)^{-\delta}
	\end{align*}
using that $\gamma<\delta/(\delta+1)$. A similar calculation for the second integral yields the same bound and as
 $|x-y|^d>(pA)^{1/\delta}\beta s^{-\gamma}t^{\gamma-1}$ implies 
 $pA(\beta^{-1}s^\gamma t^{1-\gamma}|x-y|^d)^{-\delta}\leq 1,$
 and therefore
	\[\P^p_{\x,\y}\{\x\sim\y\}=pA\big(\beta^{-1}s^\gamma t^{1-\gamma}|x-y|^d\big)^{-\delta},\]
 which proves the claim.
\end{proof} \medskip
{We now extend this result to bound the probability that the two given vertices $\x$ and $\y$ are connected through $k-1$ connectors.} That is, $\x$ and $\y$ are connected by a path of length $k$ and $\x$ and $\y$ are the two oldest vertices within the path. We denote this event by \smash{$\{\x\xleftrightarrow[\x,\y]{k}\y\}$}. 
\bigskip

\begin{lemma}\label{LemkConnection}
Let 
$\gamma\in(0,\frac{\delta}{\delta+1})$ and $\x=(x,t), \y=(y,s)$ be two Poisson points satisfying $|x-y|^d>(pA)^{1/\delta}{\gp(t,s)^{-1}}$. Then, for all $k\in\mathbb{N}$, we have
	\begin{equation}
		\P^p_{\x,\y}\{\x\xleftrightarrow[\x,\y]{k}\y\} \leq (I_\rho C_2)^{k-1}\P^p_{\x,\y}\{\x\sim\y\}, \label{EqkConnection}
	\end{equation}
where $C_2=\frac{2^{d\delta+3}\beta}{\delta(1-\gamma)-\gamma}$.
\end{lemma}

{
\begin{proof}
For $k=1$ there is nothing to show, so we assume $k\geq 2$.  
If $T$ is an unlabelled binary tree with $k-1$ vertices we denote by $X(T)$ the number of paths connecting 
$\x$ and $\y$ through $k-1$ connectors, which are associated with a labelling of $T$.  
Taking the union over all (unlabelled) binary trees on $k-1$ vertices we get	
\begin{align*}
\P^p_{\x,\y}\{\x\xleftrightarrow[\x,\y]{k}\y\} \leq  \sum_{\heap{T \text{ binary tree}}{\text{on $k-1$ vertices}}} \E^p_{\x,\y}\big[ X(T) \big],	
\end{align*}
and as the number of binary trees on $k-1$ vertices is bounded from above by\footnote{The number of binary rooted trees of size $n$ is given by the Catalan numbers $(2n)!/ (n! (n+1)!)$.} $4^{k-1}$ it suffices to show 
$$ \E^p_{\x,\y}\big[ X(T) \big] \leq (I_\rho C_1)^{k-1}\P^p_{\x,\y}\{\x\sim\y\},$$
for all binary trees $T$ with $k-1$ vertices. We show this by induction on $k$ starting with the case $k=2$, when $T$ consist of just the root,
which is shown in Lemma~\ref{LemTwoConnection}. For the induction step we fix an unlabelled binary tree $T$ with $k-1$ vertices and insert a new leaf.
Denote the new tree with $k$ vertices by $T'$.
We identify the two vertices in the tree, which correspond to the preceding and subsequent vertex of the new leaf in any path associated with $T'$ as follows:
\begin{itemize}
\item If the new leaf is a left child, its subsequent vertex in the path is its parent, and its preceding vertex is determined by following its ancestral line backwards along the tree until we find 
a vertex which has a right child on the ancestral line. If there is no such vertex its preceding vertex is $\x$.
\item  If the new leaf is a right child, its preceding vertex in the path is its parent, and its subsequent vertex is determined by following its ancestral line backwards along the tree until we find 
a vertex which has a left child on the ancestral line. If there is no such vertex its subsequent vertex is $\y$.
\end{itemize}
From the construction of the tree we make the following two observations  if a path is associated with $T'$,
\begin{itemize}
\item[(i)] the new leaf is younger than its parent and the path contains two sequential edges, 
one connecting the preceding vertex to the new leaf, and one connecting the new leaf to its  subsequent vertex,
\item[(ii)] if the preceding and subsequent vertex of the new leaf are connected by an edge, then the path using that edge instead of the
the two edges adjacent to the new leaf is associated with $T$.
\end{itemize}
We call a labelling of $T$ by points of the Poisson process \emph{almost complete} if it becomes the labelling associated with a path when 
the preceding and subsequent vertex of the new leaf are connected by an edge. Hence (ii) can be restated saying that the 
labelling of $T$ obtained by association of a path with $T'$ is almost complete. \smallskip

Denoting the labels of the preceding and subsequent vertices of the new leaf  by $\x_\ell=(x_\ell, t_\ell)$ resp.\  $\x_r=(x_r, t_r)$ we get using (i) that
\begin{align*}
\E^p_{\x,\y} & \big[ X(T') \big]  \\
& = \E^p_{\x,\y}\big[ \sharp \big\{ \text{paths } \x\leftrightarrow \x_\ell \sim \x_{\rm new}\!\sim  \x_r\leftrightarrow \y \text{ associated with }T' \big\} \big]\\
& \leq   \E^p_{\x,\y}\Big[ \sum_{\heap{\text{almost complete}}{\text{labellings of $T$}}}\int_{t_\ell\vee t_r}^1 du \int_{\R^d} dz 
\varphi((x_\ell,t_\ell),(z,u))\varphi((z,u),(x_r,t_r))\Big].
\end{align*}
As the paths associated to $T'$ are shortcut-free we have $|x_\ell-x_r|^d>(pA)^{1/\delta}\beta g^{\text{pa}}(t_\ell,t_r)^{-1}$ 
and hence Lemma~\ref{LemTwoConnection} ensures that this is bounded by
\begin{align*}
I_{\rho} \, C_1 \, & \E^p_{\x,\y}\Big[ \sum_{\heap{\text{almost complete}}{\text{labelling of $T$}}} \P^p_{\x_\ell,\x_r}\{\x_\ell\sim\x_r\} \Big] \\
& \leq I_{\rho} \, C_1 \,  \E^p_{\x,\y}\big[ X(T) \big]  \leq (I_\rho C_1)^{k}\P^p_{\x,\y}\{\x\sim\y\},
\end{align*}
using (ii) and the induction hypothesis.
\end{proof}}

\paragraph{BK-inequality} We use a version of the famous van den Berg-Kesten (BK) inequality \cite{vdBerg1996} where the application to our setting is described in detail in \cite[pp. 10-13]{Heydenreich2019LaceEx}. {For a path with given skeleton, the BK-inequality allows us to focus on the individual subpaths between any two consecutive skeleton vertices instead of considering the whole path at once.} 

For given Poisson points $\x_0,\x_1,\dots, \x_m$, we write $$\big\{\x_0\xleftrightarrow[\x_0,\x_1,\dots,\x_m]{k}\x_m\big\}$$ for the event that $\x_0$ and $\x_m$ are connected by a path of length $k$, that has skeleton $\x_0,\x_1,\dots,\x_m$. Recall that the length of a path is the number of edges on the path. {This definition is consistent with the previously introduced notation \smash{$\{\x\xleftrightarrow[\x,\y]{2}\y\}$} and  \smash{$\{\x \xleftrightarrow[\x,\y]{k}\y\}$}. We further denote by 
 \smash{$\{\x \xleftrightarrow[]{k}\y\}$} the event that $\x$ and $\y$ are connected by~a path of length k.}

Conditioned on the event that the three distinct points $\x_1,\x_2,\x_3$ are vertices of~$\G$, define $E$ to be the event that $\x_1$ is connected by a path { of length $n_1$} to $\x_2$ and $\x_2$ is connected by a path {of length $n_2$} to $\x_3$, where both paths only share $\x_2$ as a common vertex; we say that both paths \emph{occur disjointly}. 
We denote this disjoint occurrence by $\circ$ and write ${E=\{\x_1\xleftrightarrow[]{n_1}\x_2\}\circ\{\x_2\xleftrightarrow[]{n_2}\x_3\}}$. Furthermore, both events are \emph{increasing} in the following sense. Given any realization of the Poisson point process such that there is such a path between, say, $\x_1$ and $\x_2$, then there also exists such a path in any realization with additional vertices. Recall that $\P^p_{\x_1,\ldots,\x_n}$ denotes the law of $\G$ conditioned on $\x_1,\ldots,\x_n$ being vertices in $\X$. Then the BK-inequality from \cite[Theorem 2.1]{Heydenreich2019LaceEx} yields
	\begin{equation}
	\P^p_{\x_1,\x_2,\x_3}\left(\{\x_1\xleftrightarrow[]{n_1}\x_2\}\circ\{\x_2\xleftrightarrow[]{n_2}\x_3\}\right)\leq \P^p_{\x_1,\x_2}\{\x_1\xleftrightarrow[]{n_1}\x_2\}\P^p_{\x_2,\x_3}\{\x_2\xleftrightarrow[]{n_2}\x_3\}. \label{BKineq}
	\end{equation}
Next, let $S=(\x_0,\x_1,\dots,\x_m)$ be a given skeleton and recall that all paths we consider are self-avoiding. Then the event that the root $\0=\x_0$ starts a path of length $n$ that has skeleton $S$ can be written as 
	\[\{\x_0\xleftrightarrow[\x_0,\x_1,\dots,\x_m]{n}\x_m\}=\bigcup_{\substack{(n_1,\dots,n_m)\in\N^m: \\ n_1+\dots+n_m=n}}\{\x_0\xleftrightarrow[ \x_0,\x_1]{n_1}\x_1\}\circ\dots\circ\{\x_{m-1}\xleftrightarrow[ \x_{m-1},\x_m]{n_m}\x_m\}.\] 
Inductively, we derive as in~\eqref{BKineq} that
	\begin{equation}
		\P^p_{\x_0,\x_1,\dots,\x_m}\{\x_0\xleftrightarrow[\x_0,\x_1,\dots,\x_m]{n}\x_m\}\leq \sum_{ \substack{(n_1,\dots,n_m)\in\N^m: \\ n_1+\dots+n_m=n}}\prod_{j=1}^m\P^p_{\x_{j-1},\x_j}\{\x_{j-1}\xleftrightarrow[ \x_{j-1},\x_j]{n_j}\x_j\}. \label{BK}
	\end{equation}
	
\paragraph{Proof of the subcritical phase}
We now use the results of the previous paragraphs to bound the probability of a {shortcut-free} path of length $n$ existing by some exponential, thus showing Theorem \ref{ThmPercolationPhase}(a). To this end, we have to distinguish between regular and irregular paths. Let $S=(\x_0,\x_1,\dots,\x_m)$ be a skeleton of length~$m$. We say $S$ is \emph{regular} if its oldest vertex is born after time $2^{-m}$. We say $S$ is \emph{irregular} if its oldest vertex is born before time $2^{-m}$. Similarly, we say a path $P$ of finite length is regular if its underlying skeleton is regular and conversely, $P$ is irregular if its skeleton is irregular. Finally, let $P=(\v_0,\v_1,\dots)$ be an infinite path. We say $P$ is \emph{irregular} if for all $k\in\N$ there exists $n\geq k$ such that the path (of length~\(n\)) $(\v_0,\dots,\v_n)$ is irregular. An infinite path $P$ is \emph{regular} if it is not irregular. In other words, an infinite path is irregular if it has irregular subpaths of arbitrarily large lengths. We first show that almost surely any path is regular on a large enough scale, that is any irregular path becomes regular if it is extended by enough additional vertices. Therefore, $\{\0\leftrightarrow\infty\}$ equals the event that the root $\0$ starts an infinite path that is regular and we then show that no such path exists.

\begin{proof}[Proof of Theorem~\ref{ThmPercolationPhase}(a)] Observe that if an irregular path of length $n$ exists, then an irregular path of length $k\leq n$, whose end vertex is the oldest vertex of the path also exists. Let $A_\text{irreg}(k)$ be the event that $\0$ starts an irregular path of length $k$ where the end vertex is the oldest one. We will prove in the following lemma that $\P_0^p(A_\text{irreg}(k))\leq (C_3I_\rho)^k$ for some constant $C_3$. We then choose $p$ such that $I_\rho<C_3^{-1}$ and achieve 
	\[\sum_{k=1}^\infty\P_0^p(A_\text{irreg}(k))<\infty.\] 
Hence, Borel-Cantelli yields that almost surely any long enough path is  regular. \medskip

\begin{lemma}\label{LemIrregPath}
Let $\gamma\in[0,\frac{\delta}{\delta+1})$. Then, { for all $k\in \N$},
	\[\P_0^p (A_{\text{irreg}}(k))\leq (C_3 I_\rho)^k,\]
where {$C_3=2C_2=\frac{\beta 2^{d\delta+4}}{\delta(1-\gamma)-\gamma}$}. 
\end{lemma}

\begin{proof}
A path of length \(k\) whose oldest vertex is also the end vertex has a skeleton whose vertices' birth times are decreasing. Thus, we again write $\0=\x_0=(x_0,t_0)$ and have by the Mecke equation as in the proof of Lemma~\ref{LemPercolationLeq} that
	\begin{align*}
		 \P_0^p & (A_\text{irreg}(k)) 
		 \leq \sum_{m=1}^k\mathbf{E}\bigg[\sum_{\substack{(x_1,t_1),\dots,(x_m,t_m)\in\X \\ t_0>t_1>\cdots> t_m \\ t_m<2^{-m}}}\P_{\X_0}^p\Big\{(x_0,t_0)\xleftrightarrow[(x_0,t_0),\dots,(x_m,t_m)]{k}(x_m,t_m)\Big\}\bigg] \\
		&= \sum_{m=1}^k \, \int\limits_{0}^1\d t_0 \int\limits_{\substack{(\R^d\times(0,1])^m \\ t_0>t_1>\cdots> t_m \\ t_m<2^{-m}}}\bigotimes_{j=1}^m \d(x_j,t_j)\P^p_{\x_0,\dots,\x_m}\Big\{(x_0,t_0)\xleftrightarrow[(x_0,t_0),\dots,(x_m,t_m)]{k}(x_m,t_m)\Big\},
	\end{align*}
where we have written $\x_j=(x_j,t_j)$ for $j=1,\dots,m$ as usual. Using the BK-Inequality~\eqref{BK} and Lemma~\ref{LemkConnection}, we get for the last probability, 
	\begin{align*}
	&\P^p_{\x_0,\dots,\x_m}\Big\{(x_0,t_0)\xleftrightarrow[(x_0,t_0),\dots,(x_m,t_m)]{k}(x_m,t_m)\Big\} \\
	 &\quad \leq \sum_{\substack{(n_1,\dots,n_m)\in\N^m: \\ n_1+\dots+n_m=n}} \prod_{j=1}^m \P^p_{\x_{j-1},\x_j}\{(x_{j-1},t_{j-1})\xleftrightarrow[\x_{j-1},\x_j]{n_j}(x_j,t_j)\} \\
	 &\quad \leq \sum_{\substack{(n_1,\dots,n_m)\in\N^m: \\ n_1+\dots+n_m=n}} (C_2 I_\rho)^{k-m}\prod_{j=1}^m\P^p_{\x_{j-1},\x_j}\{(x_{j-1},t_{j-1})\sim(x_j,t_j)\}\\
	 &\quad { = {k-1\choose m-1}(C_2 I_\rho)^{k-m}\prod_{j=1}^m\P^p_{\x_{j-1},\x_j}\{(x_{j-1},t_{j-1})\sim(x_j,t_j)\}}.
	\end{align*}
Here, we used that either the consecutive skeleton vertices $\x_{i-1}$ and $\x_i$ fulfil the minimum distance for shortcut-free paths or {$n_i=1$}. 
Therefore,
\begin{align}
	 & \P_0^p(A_\text{irreg}(k)) \label{hier} \\
	 &\leq \sum_{m=1}^k { k-1 \choose m-1}(C_2I_\rho)^{k-m} \notag \\
	 & \qquad\times \int_{0}^1\d t_0\int_{0}^{t_0}\d t_1\int_{\R^d}\d x_1\dots\int_{0}^{2^{-m}\wedge t_{m-1}} \d t_m\int_{\R^d}\d x_m \Big(\prod_{i=1}^m\rho(\beta^{-1}t_{i-1}^{1-\gamma}t_i^{\gamma}|x_{i-1}-x_i|^d)\Big) \notag \\
	 &\leq \sum_{m=1}^k { k-1 \choose m-1} C_2^{k-m} I_\rho^k \beta^{m}\int_{0}^1\d t_0\int_{0}^{t_0}\d t_1\dots \int_{0}^{2^{-m}\wedge t_{m-1}}\d t_m \ t_0^{\gamma-1}t_m^{-\gamma}\prod_{i=1}^{m-1} t_i^{-1} \notag\\
	 &\leq I_\rho^k \sum_{m=1}^k { k-1 \choose m-1} \beta^m C_2^{k-m} (1-\gamma)^{-m}  \leq (C_2I_\rho)^k{ \sum_{m=1}^k {k-1 \choose m-1} \leq (C_3 I_\rho)^k},\notag
\end{align} 
where the {third} inequality follows from Lemma~\ref{IntOutdegree}. 
\end{proof}
\pagebreak[3]

The previous lemma shows that for $I_\rho<C_3^{-1}$, it suffices to show that $\0$ does not start an infinite path that is regular in order to obtain $\theta(p)=0$. Let $A_\text{reg}(n)$ be the event that $\0$ starts a regular path of length $n$. \pagebreak[3]
	
\begin{lemma}\label{LemRegPath}
Let $\gamma\in[\frac{1}{2},\frac{\delta}{\delta+1})$. Then, {for 
all $n\in\N$}, we have
	\[\P_0^p(A_{\text{reg}}(n))\leq K(C_3 I_\rho)^n,\]
where $C_3=2C_2=\frac{\beta 2^{d\delta+4}}{\delta(1-\gamma)-\gamma}$ and $K$ is some constant.
\end{lemma}
 \pagebreak[3]

\begin{proof}
Writing $\0=\x_0=(x_0,t_0)$ and following the same arguments of Mecke equation, BK-Inequality and Lemma~\ref{LemkConnection} as done in the previous proof of Lemma~\ref{LemIrregPath}, we get for large enough $n$ that
	\begin{align}
	 \P^p_0(A_\text{reg}(n))\leq &  \sum_{m=1}^n  \sum_{k=0}^m \int_{2^{-m}}^1\d t_0 \  {n-1 \choose m-1}(C_2 I_\rho)^{n-m} \notag \\
	 											 & \times\int\limits_{\substack{(x_1,t_1),\dots(x_m,t_m)\in\R^d\times(0,1] \\ t_0>t_1>\dots >t_k>2^{-m} \\ t_k<t_{k+1}<\dots <t_m}} \bigotimes_{j=1}^m \d(x_j,t_j) \ \prod_{j=1}^m \varphi((x_{j-1},t_{j-1}),(x_j,t_j)). \label{eqRegPath}
	\end{align}
Here, the two sums and integrals describe all regular skeletons a regular path of length~$n$ can have. For the calculation, we focus on $\gamma>1/2$. For $\gamma=1/2$ minor changes are needed; we comment on this below. Recall that 
	\[\varphi((x_{j-1},t_{j-1}),(x_j,t_j))=\rho({\gp(t_{j-1}, t_j)}|x_{j-1}-x_j|^d).\]
Therefore, the right-hand side of \eqref{eqRegPath} reads
	\begin{align}
		& \sum_{m=1}^n {n-1 \choose m-1} (C_2 I_\rho)^{n-m}\notag \\ 
		& \qquad \times\sum_{k=0}^m I_\rho^m \int\limits_{\substack{1>t_0>t_1>\dots >t_k>2^{-m} \\ t_k<t_{k+1}<\dots <t_m}} \bigotimes_{j=0}^m\d t_j \ \prod_{j=1}^m {\gp(t_{j-1},t–j)^{-1}}. \label{eqRegPath2}
	\end{align}
For $k=0$ the integral from \eqref{eqRegPath2} can be written as
	\[\beta^m\int_{2^{-m}}^1\d t_0\int_{t_0}^1\d t_1 \dots\int_{t_{m-1}}^1\d t_m t_0^{-\gamma}t_m^{\gamma-1}\prod_{j=1}^{m-1} t_j^{-1}\leq \left(\frac{\beta}{1-\gamma}\right)^m,\]
by Lemma~\ref{IntIndegree}. For $k=m$, we obtain for the integral from \eqref{eqRegPath2}
	\[\beta^m\int_{2^{-m}}^1\d t_0\int_{2^{-m}}^{t_0} \d t_1\dots \int_{2^{-m}}^{t_{m-1}}d t_m t_0^{\gamma-1}t_m^{-\gamma}\prod_{j=1}^{m-1}t_j^{-1}\leq \left(\frac{\beta}{1-\gamma}\right)^m,\]
by Lemma~\ref{IntOutdegree}.
For $1\leq k\leq m-1$, we infer for the integral from \eqref{eqRegPath2}, using Lemma~\ref{IntSkeleton}, 

	\begin{align*}
		& \beta^m\sum_{k=1}^{m-1} \int_{2^{-m}}^1\d t_0\int_{2^{-m}}^{t_0}\d t_1\dots \int_{2^{-m}}^{t_{k-1}}\d t_k \Bigg[t_0^{\gamma-1}\left(\prod_{j=1}^{k-1}t_j^{-1}\right)t_k^{-\gamma} \\
		& \hspace{5cm}\times\int_{t_k}^1\d t_{k+1}\dots\int_{t_{m-1}}^1\d t_m \left[t_k^{-\gamma}\left(\prod_{j=k+1}^{m-1} t_j^{-1}\right)t_m^{\gamma-1}\right]\Bigg]	\\
		& \qquad \leq  \beta^m\frac{2^{-m(1-2\gamma)}(m\log(2))^{m-2}}{\gamma^2(2\gamma-1)(m-2)!}\sum_{k=1}^{m-1} {{m-2}\choose {k-1}}.
	\end{align*}
Since $m^{m-2}/(m-2)!$ asymptotically equals $2^{\log_2(e)(m-2)}/\sqrt{2\pi(m-2)}$ by Stirling's formula, and 
\smash{$\sum_{k=1}^{m-1}{ {m-2}\choose {k-1}}\leq 2^{m-2}$}, we infer from \eqref{eqRegPath} and \eqref{eqRegPath2}
	\begin{align*}
	 \P_0^p (A_\text{reg}(n)) & \leq I_\rho^n K\sum_{m=1}^n {n-1 \choose m-1} \beta^m C_2^{n-m}\left((1-\gamma)^{-m}+(2^{2\gamma+\log_2(e)}\log(2))^m\right),
	\end{align*}
for some constant $K\geq 2$.
As $C_2>(1-\gamma)^{-1}$ and $C_2\geq 2^{2\gamma+\log_2(e)}\log(2)$ we infer that
	\[\P_0^p(A_\text{reg}(n))\leq K(I_\rho C_3)^n.\]
For $\gamma=\frac12$, Lemma~\ref{IntTwoGamma} and Lemma~\ref{IntSkeleton} have to be modified slightly. The changes in the calculations only influence the value of $K$ and not the constant~$C_3$.
\end{proof} 

{ Setting \(p\) small enough that \(C_3I_{\rho}<1\) concludes the proof of \Cref{ThmPercolationPhase}(a).}
\end{proof}

\section{Absence of a subcritical phase}\label{SecSupercrit}

In this section, we prove Theorem~\ref{ThmPercolationPhase}(b) using a strategy of Jacob and M\"orters~\cite{JacobMoerters2017}. Starting from a sufficiently old vertex, we use a young connector to connect the old vertex with a much older one; we repeat this indefinitely, moving to older and older vertices as we go along. To ensure that this procedure generates an infinite path with positive probability, we have to show that the failure probabilities of connecting the pairs of increasingly old vertices sum to a probability strictly less than one.
\smallskip

 To this end, we show that an old vertex is \emph{with extreme probability} connected to a much older one by a single connector. Here, if $(A(t))_{t>0}$ is a family of events, we say an event $A(t)$ holds with extreme probability, or $wep(t)$, if it holds with probability at least $1-\exp(-\Omega(\log^2(t)))$, as $t\to\infty$, where $\Omega(t)$ is the standard Landau symbol. Observe, if $(A(t)_n)_{n\in\N}$ is a sequence of events, holding simultaneously $wep(t)$ in the sense that 
$$\inf_n \P(A(t)_n)\geq 1-\exp(\Omega(\log^2(t))),$$ as $t \to \infty$, then $\bigcap_{k\leq \lfloor t\rfloor}A(t)_k$ 
holds~$wep(t)$. \smallskip

 Because $g^{\text{pa}}, g^\text{sum}\leq g^\text{min}$ we can fix the kernel $g$ to be the min kernel $g^\text{min}$ throughout this 
 section.\footnote{ We retain the terminology of \emph{old} and \emph{young} vertices motivated by preferential attachment for convenience and better comparison with 
 the previous section.} Hence, for two given vertices $\x=(x,t)$ and $\y=(y,s)$, the connection probability is given by
	\[\varphi(\x,\y)=p \rho(\beta^{-1}(s\wedge t)^{-\gamma}|x-y|^d).\]  	
Recall that $\rho$ is regularly varying with index $-\delta$ for $\delta>1$. Further, $\gamma>\delta/(\delta+1)$. Thus, we can choose
	\[\alpha_1\in\left(1,\tfrac{\gamma}{\delta(1-\gamma)}\right) \text{ and then fix }\alpha_2\in\left(\alpha_1,\tfrac{\gamma}{\delta}(1+\alpha_1\delta)\right).\]
The following lemma shows that the outlined strategy for an infinite path works and thus proves Theorem~\ref{ThmPercolationPhase}(b).
\medskip

\begin{lemma}\label{LemRobustness}
Let \smash{$\gamma>\frac{\delta}{\delta+1}$} and $\rho$ be regularly varying with index $-\delta$ for $\delta>1$. Let $\alpha_1,\alpha_2$ be as defined as above. Let $\x_0=(x_0, s_0)$ be a given Poisson point with $s_0<\frac{1}{2}$. Then, for any retention parameter $p>0$, $wep(1/s_0)$, there exists a sequence $(\x_k)_{k\in\N}$ of vertices $\x_k=(x_k,s_k)\in \X$ such that
\begin{enumerate}[(i)]
	\item $s_k<s_{k-1}^{\alpha_1}$ and $|x_k-x_{k-1}|^d<\frac{\beta}{2} s_{k-1}^{-\alpha_2}$ and
	\item $\x_{k-1}\xleftrightarrow[\x_{k-1},\x_k]{2}\x_k$
\end{enumerate}
for all $k\in\N$.
\end{lemma}

\pagebreak[3]
\begin{proof}
It suffices to show that, $wep(1/s_0)$, there exists a vertex $\x_1=(x_1,s_1)$ satisfying (i) and (ii). The result then follows by induction. The number of vertices, born before time $s_0^{\alpha_1}$ and within distance $((\beta/2) s_0^{-\alpha_2})^{1/d}$ from $x_0$ is Poisson distributed with parameter 
	\[\text{Vol}\big(\{|x_1-x_0|^d<\tfrac{\beta}{2} s_0^{-\alpha_2}\}\times (0,s_0^{\alpha_1})\big)=O(s_0^{\alpha_1-\alpha_2}),\]
where $O(\cdot)$ again is the standard Landau symbol. Since $\alpha_2>\alpha_1$, there exists such vertex $\x_1$, $wep(1/s_0)$. To connect $\x_0$ to $\x_1$ via a young vertex, we focus on connectors $(y,t)$, born after time $1/2$ and within distance $((\beta/2) s_0^{-\gamma})^{1/d}$ from $\x_0$. Since, for such choices of $(y,t)$, we have
	\[|x_1-y|^d\leq\big((\tfrac{\beta s_0^{-\alpha_2}}{2})^{1/d}+(\tfrac{\beta s_0^{-\gamma}}{2})^{1/d}\big)^d\leq \beta s_0^{-\alpha_2},\] 
the number of such connectors is Poisson distributed with its parameter bounded from below by
	\begin{align}
		& p^2\int_{1/2}^1 \d t\int\limits_{\{|y-x_0|^d\leq \frac{\beta}{2}s_0^{-\gamma}\}} \d y \ \rho(\beta^{-1}s_0^{\gamma}|y-x_0|^{-d})\rho(s_0^{\alpha_1\gamma-\alpha_2}) \notag \\
		&\quad= p^2\tfrac{1}{2} \beta s_0^{-\gamma} \rho(s_0^{\alpha_1\gamma-\alpha_2}) \int\limits_{\{|y-x_0|^d\leq 1/2\}} \d y  \ \rho(|y-x_0|^d). \label{eqIntRobustness}
	\end{align}
Now, we choose $\varepsilon>0$ such that $\tilde{\delta}:=\delta+\varepsilon<\frac{\gamma}{1-\gamma}$, or equivalently $\gamma>\tilde{\delta}/(\tilde{\delta}+1)$, and infer by the Potter bound \cite[Theorem 1.5.6]{Bingham1987},
	\[\rho(s_0^{\alpha_1\gamma-\alpha_2})\geq A s_0^{-\tilde{\delta}(\alpha_1\gamma-\alpha_2)},\] 
for some $A<1$ and $s_0$ small enough. Additionally, $\rho(|x|^d)\geq\rho(1/2)>0$ for all $|x|^d<1/2$. Hence, \eqref{eqIntRobustness} is bounded from below by 
	\[\Omega\big(s_0^{-\tilde{\delta}(\alpha_1\gamma-\alpha_2)-\gamma}\big).\]
Therefore, $wep(1/s_0)$, $\x_1$ satisfies (ii) as
	\[\P^p_{\x_0,\x_1}\{\x_0\xleftrightarrow[\x_0,\x_1]{2}\x_1\}\geq 1-\exp(-\Omega(s_0^{-\tilde{\delta}(\alpha_1\gamma-\alpha_2)-\gamma}))\]
and $-\tilde{\delta}(\alpha_1\gamma-\alpha_2)-\gamma<0$.
\end{proof}

\section{Proof of Theorem 1.2} \label{SecRobustness}

We first introduce for finite $t>0$ the rescaling map
\begin{center}
	\begin{tabular}{llll}
		$h_t:$ & $\mathbb{T}_1^d\times(0,t]$ & $\longrightarrow$ & $\mathbb{T}_t^d\times (0,1]$, \\
			   & $(x,s)$					 & $\longmapsto$	     & $\left(t^{1/d}x, s/t\right)$.
	\end{tabular}
\end{center}
It gives rise to a new graph $h_t(\G_t^p)$ whose vertices live on $\mathbb{T}_t^d\times(0,1]$ and where two rescaled vertices are connected in $h_t(\G_t^p)$ if they were originally connected in~$\G_t^p$. It is easy to see that
$h_t(\G_t^p)$ is the graph  with vertex set given by a standard Poisson process on $\mathbb{T}_t^d\times (0,1]$ and
independent edges with the same connection probability as in~\eqref{PAProb}, see \cite{GracarEtAl2019}. 
The process $t\mapsto h_t(\G_t^p)$  converges almost surely to the graph $\G^p$  in the sense that if a randomly selected point in $h_t(\G_t^p)$ is shifted to the origin, the embedded graph in any ball around the origin converges in distribution as $t\to\infty$,  to the same ball centred in the origin of~$\G_0^p$, see \cite[Theorem 3.1]{GracarEtAl2019}. To obtain the weak local limit structure for the age-based preferential attachment network, 
let $h_t^0(\G_t^p)$ be the graph $h_t(\G_t^p)$  with a root vertex $\0$ added at the origin. If $G$ is a locally finite graph equipped with a root $\x\in G$ and $\xi_t(\x, G)$ is a  non-negative functional acting on rooted graphs that satisfy
\begin{enumerate}[(A)]
	\item $\xi_t(\0, h_t^0(\G_t^p))\rightarrow \xi_\infty(\0, \G_0^p)$ in probability as $t\to\infty$ and
	\item $\sup_{t>0}\E[(\xi_t(\0, h_t^0(\G_t^p)))^q]<\infty$ for some $q>1$,
\end{enumerate}
then we get from Theorem~7 of \cite{JacobMoerters2015}, 
	\begin{equation}
		\lim_{t\to\infty} \frac{1}{t} \sum_{\x\in \G_t^p} \xi_t(\theta_\x(\x),\theta_\x(\G_t^p))=\E[\xi_\infty(\0, \G_0^p)] \label{WLLN}
	\end{equation}
in probability, where $\theta_\x$ acts on points $\y=(y,s)$ as $\theta_\x(\y)=(y-x,s)$ and on graphs accordingly. This weak law of large numbers is an adaptation of a general weak law of large numbers for point processes of Penrose and Yukich \cite{PenroseYukich2003}. \smallskip

For the proof of non-robustness in Theorem~\ref{ThmRobustness} define $\xi^k(\x,G)$ as indicator that the component of the root vertex~$\x$ is of size at most $k$. By the weak law of large numbers 
	\[\lim_{t\to\infty}\frac{1}{t}\sum_{\x\in \G_t^p}\xi^k(\theta_\x(\x),\theta_\x(\G_t^p))=\E[\xi^k(\0,\G^p)].\]
	\ \\[-5mm]
The left hand side is asymptotically the proportion of vertices that are in components no bigger than $k$. As $k\to\infty$, the right hand side converges to $1-\theta(p)$ and  if we choose a $p>0$ such that $\theta(p)=0$, 
there is no giant component in  $(\G_t^p)_{t>0}$.
\medskip
\pagebreak[3]

For the proof of robustness in Theorem~\ref{ThmRobustness} define $\xi_t(\x, G)$ as indicator that the root $\x$ of $G$ belongs to the connected component of the oldest vertex in the finite graph $G$, and $\xi_\infty(\0, G)$ as the indicator that the root $\0$ of $G$ belongs to an infinite component in the infinite graph $G$. Then one has to show that 
$$\xi_t(\0, h_t^0(\G_t^p))\to \xi_\infty(\0, \G_0^p) \mbox{ in probability as $t\to\infty$. }$$
This 
is done in detail in \cite{JacobMoerters2017} for the spatial preferential attachment model and can be easily adapted 
to the simpler age-based preferential attachment model. The weak law of large numbers then yields
\[\lim_{t\to\infty}\frac{1}{t}\sum_{\x\in \G_t^p}\xi_t(\theta_\x(\x),\theta_\x(\G_t^p))=\E[\xi_\infty(\0,\G^p)]=\theta(p)\]
in probability. Again, we see from this that if $\theta(p)>0$ there is a giant component and the result follows from Theorem~\ref{ThmPercolationPhase}, further details are exactly as in \cite{JacobMoerters2017}. 
\medskip
\pagebreak[3]

\appendix
\section{Integration results} \label{appendix}
\begin{lemma}\label{IntIndegree}
Let $\gamma\in(0,1)$ and $t_0\in(0,1)$. Then,
\begin{enumerate}[(a)]
	\item for all $k\in\mathbb{N}$, we have
		\[\int_{t_0}^1\d t_1\int_{t_1}^1 \d t_2\dots \int_{t_{k-1}}^1\d t_k \bigg[ t_0^{-\gamma}\Big(\prod_{j=1}^{k-1}t_j^{-1}\Big)t_k^{\gamma-1}\bigg]\leq \frac{t_0^{-\gamma}\log^{k-1}(1/t_0)}{\gamma(k-1)!}.\]
	\item for all $k\in\mathbb{N}$, we have
		\[\int_0^1 \d t\frac{t^{-\gamma}\log^k(1/t)}{k!}=\left(\frac{1}{1-\gamma}\right)^{k+1}.\]
\end{enumerate}
\end{lemma}
\begin{proof}
We prove (a) by induction. For $k=1$, we have
	\[t_0^{-\gamma}\int_{t_0}^1\d t_1 t_1^{\gamma-1}\leq \frac{t_0^{-\gamma}}{\gamma}.\]
For $k+1$ we get using the induction hypothesis
	\begin{align*}
	 & \int_{t_0}^1\d t_1\int_{t_1}^1 \d t_2\dots \int_{t_{k}}^1\d t_{k+1} \bigg[ t_0^{-\gamma}\Big(\prod_{j=1}^{k}t_j^{-1}\Big)t_{k+1}^{\gamma-1}\bigg] 
	  \leq t_0^{-\gamma}\int_{t_0}^1 \d t_1 \frac{t_1^{-1}\log^{k-1}(1/t_1)}{\gamma(k-1)!} \\
	  & \quad= \frac{t_0^{-\gamma}(-1)^{k-1}}{\gamma(k-1)!}\int_{t_0}^1 \d t_1\log(t_1)'\log^{k-1}(t_1) 
	  = \frac{t_0^{-\gamma}\log^k(1/t_0)}{\gamma k!}.
	\end{align*}
We prove (b) by induction as well. As $\gamma<1$, we get, for $k=1$ using integration by parts
	\[\int_0^1 \d t \frac{t^{-\gamma}\log(1/t)}{1!}=\int_0^1\d t \frac{t^{-\gamma}}{1-\gamma}=\frac{1}{(1-\gamma)^2}.\]
Analogously for $k+1$,
 \[\int_0^1\d t \frac{t^{-\gamma}\log^{k+1}(1/t)}{(k+1)!}=\int_0^1 \d t\frac{t^{-\gamma}\log^k(1/t)}{(1-\gamma)k!}=\frac{1}{(1-\gamma)^{k+2}}\]
by the induction hypothesis.
\end{proof}

\begin{lemma}\label{IntTwoGamma} 
Let $\gamma\in(1/2,1)$ and $x\in(0,1)$. Then, for all $k\in\mathbb{N}$, it holds
	\[\int_x^1 \d t \ \frac{t^{-2\gamma}\log^k(1/t)}{k!}\leq \frac{x^{1-2\gamma}\log^k(1/x)}{(2\gamma-1)k!}.\]
\end{lemma}
\begin{proof}
Integration by parts yields
	\begin{align*}
	\int_x^1 \d t \ \frac{t^{-2\gamma}\log^k(1/t)}{k!} =\frac{x^{1-2\gamma}\log^k(1/x)}{(2\gamma-1)k!} -\int_x^1\d t\frac{t^{-2\gamma}\log^{k-1}(1/t)}{(2\gamma-1)(k-1)!}\leq \frac{x^{1-2\gamma}\log^k(1/x)}{(2\gamma-1)k!},
	\end{align*}
as the second integral is bounded from below by $0$.
\end{proof}

\begin{lemma}\label{IntDecreasing}
Let $\gamma\in(0,1)$, $x\in(0,1)$ and $t_0\in(x,1)$. Then, for all $k\in\mathbb{N}$, it holds
	\[\int_x^{t_0}\d t_1\int_x^{t_1}\d t_2\dots\int_x^{t_{k-1}}\d t_k \Big( t_0^{\gamma-1}\prod_{j=1}^k t_j^{-1}\Big)=\frac{t_0^{\gamma-1}\log^k(t_0/x)}{k!}.\]
\end{lemma}
\begin{proof}
For $k=1$, we get
	\[\int_x^{t_0}\d t_1 \ t_0^{\gamma-1}t_1^{-1}=t_0^{\gamma-1}\log(t_0/x).\]
For $k+1$, using induction hypothesis, we get
	\begin{align*}
	\int_x^{t_0}\d t_1\int_x^{t_1}\d t_2\dots\int_x^{t_{k}}\d t_{k+1}& \Big( t_0^{\gamma-1}\prod_{j=1}^{k+1} t_j^{-1}\Big)
	 = t_0^{\gamma-1}\int_x^{t_0} \d t_1 \frac{t_1^{-1}\log^k(t_1/x)}{k!} \\
	& =t_0^{\gamma-1}\int_0^{\log(t_0/x)}\d y \frac{y^k}{k!}  = \frac{t_0^{\gamma-1}\log^{k+1}(t_0/x)}{(k+1)!}.
	\end{align*}
\end{proof}

\begin{lemma}\label{IntSkeleton}
Let $\gamma\in(1/2,1)$ and $m,k\in\mathbb{N}$, such that $m\geq 2$ and $1\leq k\leq m-1$. Further, let $x\in(0,1)$. Then,
	\begin{align}
		&\int\limits_x^1\d t_0\int\limits_x^{t_0}\d t_1\dots \int\limits_x^{t_{k-1}}\d t_k \left[t_0^{\gamma-1}\Big(\prod_{j=1}^{k-1}t_j^{-1}\Big)t_k^{-\gamma}\int\limits_{t_k}^1\d t_{k+1}\dots\int\limits_{t_{m-1}}^1\d t_m \bigg[t_k^{-\gamma}\Big(\prod_{j=k+1}^{m-1} t_j^{-1}\Big)t_m^{\gamma-1}\bigg]\right] \notag \\
		& \ \leq {{m-2}\choose {k-1}}\frac{x^{1-2\gamma}\log^{m-2}(1/x)}{\gamma^2(2\gamma-1)(m-2)!}. \label{intAppendix}
	\end{align}
\end{lemma}
\begin{proof}
We apply the previous lemmas. By Lemma~\ref{IntIndegree}, we get
	\begin{align*}
	\int_{t_k}^1\d t_{k+1}\dots\int_{t_{m-1}}^1\d t_m \bigg[t_k^{-\gamma}\Big(\prod_{j=k+1}^{m-1} t_j^{-1}\Big)t_m^{\gamma-1}\bigg] \leq \frac{t_k^{-\gamma}\log^{m-k-1}(1/t_k)}{\gamma(m-k-1)!}.
	\end{align*}
Therefore, the integral in \eqref{intAppendix} can be bound by
	\begin{align*}
		\int_x^1 \d t_0\int_x^{t_0}\d t_1\dots \int_x^{t_{k-2}}\d t_{k-1} \bigg[t_0^{\gamma-1}\Big(\prod_{j=1}^{k-1}t_j^{-1}\Big)\int_{x}^{t_{k-1}} \d t_k \frac{t_k^{-2\gamma}\log^{m-k-1}(1/{t_k})}{\gamma(m-k-1)!}\bigg].
	\end{align*}
By Lemma~\ref{IntTwoGamma}
	\[ \int_x^{t_{k-1}}\d t_k \frac{t_k^{-2\gamma}\log^{m-k-1}(1/{t_k})}{\gamma(m-k-1)!} \leq \frac{x^{1-2\gamma}\log^{m-k-1}(1/x)}{\gamma(2\gamma-1)(m-k-1)!}\]
and by Lemma~\ref{IntDecreasing}
	\[\int_x^{t_0}\d t_1\dots \int_x^{t_{k-2}}\d t_{k-1} \ t_0^{\gamma-1}\Big(\prod_{j=1}^{k-1}t_j^{-1}\Big)= \frac{t_0^{\gamma-1}\log^{k-1}(t_0/x)}{(k-1)!}.\]
Therefore, the integral in \eqref{intAppendix} can be further bound by
	\begin{align*} & \int_x^1 \d t_0\frac{t_0^{\gamma-1}\log^{k-1}(t_0/x)}{(k-1)!}\frac{x^{1-2\gamma}\log^{m-k-1}(1/x)}{\gamma(2\gamma-1)(m-k-1)!} \\
	& \qquad \leq {{m-2}\choose {k-1}}\frac{x^{1-2\gamma}\log^{m-2}(1/x)}{\gamma(2\gamma-1)(m-2)!}\int_x^1 \d t_0 \ t_0^{\gamma-1}. 
	\end{align*}
The result follows by integrating with respect to $t_0$. 
\end{proof}

\begin{lemma}\label{IntOutdegree}
Let $\gamma\in(0,1)$ and $k\in\N$. Then
	\[\int_0^1\d t_0\int_0^{t_0}\d t_1\dots\int_0^{t_{k-1}}\d t_k \ t_0^{\gamma-1}\bigg(\prod_{j=1}^{k-1}t_j^{-1}\bigg)t_k^{-\gamma}\leq \bigg(\frac{1}{1-\gamma}\bigg)^k.\] 
\end{lemma}
\begin{proof}
We have
	\begin{align*}
	\int_0^1\d t_0 & \int_0^{t_0}\d t_1\dots\int_0^{t_{k-2}}\d t_{k-1}\bigg[t_0^{\gamma-1}\Big(\prod_{j=1}^{k-1}t_j^{-1}\Big)\int_0^{t_{k-1}}\d t_{k} \ t_k^{-\gamma}\bigg] \\
	& = \frac{1}{1-\gamma}\int_0^1\d t_0\int_0^{t_0}\d t_1\dots\int_0^{t_{k-2}}\d t_{k-1} \ t_0^{\gamma-1}\Big(\prod_{j=1}^{k-2}t_j^{-1}\Big)t_{k-1}^{-\gamma}
	\end{align*}
and the result follows by repeating this across all integrals.
\end{proof}\pagebreak[3]

\end{spacing}

\smallskip

{\bf Acknowledgements:} This research was supported by Deutsche Forschungsgemeinschaft (DFG) Project 425842117 and forms part of 
the second author's PhD thesis. We would like to thank a referee and an associate editor for their helpful comments.%
\smallskip%

\footnotesize{\printbibliography}
\typeout{get arXiv to do 4 passes: Label(s) may have changed. Rerun}
\end{document}